\documentclass{tac}
\usepackage[utf8]{inputenc}
\usepackage{amsmath}
\usepackage{amssymb}
\usepackage{stmaryrd}
\usepackage{mathrsfs}
\usepackage{enumitem}
\usepackage{xurl}
\usepackage[colorlinks=true]{hyperref}
\hypersetup{allcolors=[rgb]{0.1,0.1,0.4}}

\usepackage{tikz}
\usetikzlibrary{arrows,arrows.meta,decorations.pathreplacing,decorations.markings,shapes,calc}
\tikzset{labelsize/.style={font=\scriptsize}}
\tikzset{string/.style={very thick}}
\tikzset{
  pto/.style={->,postaction={decorate},
    decoration={
        markings,
        mark=at position 0.5 with {\arrow{|}}}
  },
}
\tikzset{2cell/.style={-implies,double,double equal sign distance,shorten >=9pt, shorten <=10pt}}

\author{Soichiro Fujii and Stephen Lack}
\address{School of Mathematical and Physical Sciences, Macquarie University, NSW 2109, Australia}
\thanks{The authors acknowledge with gratitude the support of an Australian Research Council Discovery Project DP190102432. The first author is a JSPS Overseas Research Fellow.}
\dedication{In memory of our colleague Marta Bunge}
\title{The oplax limit of an enriched category}
\copyrightyear{2024}

\eaddress{s.fujii.math@gmail.com\CR steve.lack@mq.edu.au}

\mathchardef\mhyphen="2D
\newcommand{\cat}[1]{\mathcal{#1}}
\newcommand{\bcat}[1]{\mathscr{#1}}
\newcommand{\ecat}[1]{\mathbb{#1}}

\newcommand{\enCat}[1]{{#1}\mhyphen\mathbf{Cat}}

\newcommand{\op}{\mathrm{op}}
\newcommand{\Set}{\mathbf{Set}}
\newcommand{\Setlc}{\mathbf{Set}_{lc}}
\newcommand{\SET}{\mathbf{SET}}
\newcommand{\Ab}{\mathbf{Ab}}
\newcommand{\Cat}{\mathbf{Cat}}
\newcommand{\CAT}{\mathbf{CAT}}
\newcommand{\Bicat}{\mathbf{BICAT}}
\newcommand{\Enr}{\mathrm{Enr}}
\newcommand{\twoCAT}{2\mhyphen\mathbf{CAT}}
\newcommand{\one}{\mathbf{1}}
\newcommand{\two}{\mathbf{2}}
\newcommand{\twocell}{\mathbf{2_2}}
\newcommand{\Mon}{\mathbf{Mon}}
\newcommand{\ob}{\mathrm{ob}}
\newcommand{\colim}{\mathrm{colim}}

\begin{document}
\maketitle
\begin{abstract}
We show that 2-categories of the form $\mathscr{B}\mbox{-}\mathbf{Cat}$ are closed under slicing, provided that we allow $\mathscr{B}$ to range over bicategories (rather than, say, monoidal categories). That is, for any $\mathscr{B}$-category $\mathbb{X}$, we define a bicategory $\mathscr{B}/\mathbb{X}$ such that $\mathscr{B}\mbox{-}\mathbf{Cat}/\mathbb{X}\cong (\mathscr{B}/\mathbb{X})\mbox{-}\mathbf{Cat}$. The bicategory $\mathscr{B}/\mathbb{X}$ is characterized as the oplax limit of $\mathbb{X}$, regarded as a lax functor from a chaotic category to $\mathscr{B}$, in the 2-category $\mathbf{BICAT}$ of bicategories, lax functors and icons. We prove this conceptually, through limit-preservation properties of the 2-functor $\mathbf{BICAT}\to 2\mbox{-}\mathbf{CAT}$ which maps each bicategory $\mathscr{B}$ to the 2-category $\mathscr{B}\mbox{-}\mathbf{Cat}$. When $\mathscr{B}$ satisfies a mild local completeness condition, we also show that the isomorphism $\mathscr{B}\mbox{-}\mathbf{Cat}/\mathbb{X}\cong (\mathscr{B}/\mathbb{X})\mbox{-}\mathbf{Cat}$ restricts to a correspondence between fibrations in $\mathscr{B}\mbox{-}\mathbf{Cat}$ over $\mathbb{X}$ on the one hand, and $\mathscr{B}/\mathbb{X}$-categories admitting certain powers on the other.
\end{abstract}

\section{Introduction}
It is well-known that for any monoidal category $\bcat{V}$ and monoid 
$M=(M,e\colon I\to M,$ $m\colon M\otimes M\to M)$
therein, the slice category $\bcat{V}/M$ has a canonical monoidal structure; the unit is $e$ and the monoidal product of objects  $(s\colon S\to M)$ and $(t\colon T\to M)$ is 
\[
    \begin{tikzpicture}[baseline=-\the\dimexpr\fontdimen22\textfont2\relax ]
      \node(0) at (-0.5,0) {$S\otimes T$};
      \node(1) at (2,0) {$M\otimes M$};
      \node(2) at (4,0) {$M$.};      

      \draw [->] (0) to node[auto,labelsize] {$s\otimes t$} (1);
      \draw [->] (1) to node[auto,labelsize] {$m$} (2);
    \end{tikzpicture}
\]
Moreover, there is a canonical isomorphism of categories 
\[
\Mon(\bcat{V}/M)\cong \Mon(\bcat{V})/M.
\]

This paper originated from a natural generalization of this, replacing the notion of monoid in $\bcat{V}$ by that of $\bcat{V}$-category. 
That is, for any $\bcat{V}$-category $\ecat{X}$, there is an appropriate ``base'' $\bcat{V}/\ecat{X}$ admitting a canonical isomorphism of 2-categories 
\begin{equation}\label{eqn:cat_slice}
    \enCat{(\bcat{V}/\ecat{X})}\cong\enCat{\bcat{V}}/\ecat{X}.
\end{equation}
Here, the ``base'' $\bcat{V}/\ecat{X}$ is in general not a monoidal category but a bicategory. 
Enriched category theory over bicategories is developed in, e.g., \cite{Betti-et-al-variation,Street-cohomology}. We recall that, for a bicategory $\bcat{B}$, a $\bcat{B}$-category $\ecat{X}$ is given by
\begin{itemize}
    \item a set $\ob(\ecat{X})$;
    \item a function $|{-}|\colon \ob(\ecat{X})\to \ob(\bcat{B})$ ($|x|$ is called the \emph{extent} of $x$);
    \item for all $x,x'\in \ob(\ecat{X})$, a 1-cell $\ecat{X}(x,x')\colon |x|\to |x'|$ in $\bcat{B}$;
    \item for all $x\in \ob(\ecat{X})$, a 2-cell 
    \[
    \begin{tikzpicture}[baseline=-\the\dimexpr\fontdimen22\textfont2\relax ]
      \node(0) at (0,0) {$|x|$};
      \node(1) at (3,0) {$|x|$};
      
      \draw [->,bend left=20] (0) to node[auto,labelsize] {$1_{|x|}$} (1);
      \draw [->,bend right=20] (0) to node[auto,swap,labelsize] {$\ecat{X}(x,x)$} (1);
      \draw [2cell] (1.5,0.5) to node[auto,labelsize] {$j_x$} (1.5,-0.5);
    \end{tikzpicture}    
    \]
    in $\bcat{B}$, where $1_{|x|}$ is the identity 1-cell on $|x|$; and
    \item for all $x,x',x''\in\ob(\ecat{X})$, a 2-cell 
    \[
    \begin{tikzpicture}[baseline=-\the\dimexpr\fontdimen22\textfont2\relax ]
      \node(0) at (0,0) {$|x|$};
      \node(1) at (2,0.5) {$|x'|$};
      \node(2) at (4,0) {$|x''|$};
      
      \draw [->,bend left=10] (0) to node[auto,very near end,labelsize] {$\ecat{X}(x,x')$} (1);
      \draw [->,bend left=10] (1) to node[auto,very near start,labelsize] {$\ecat{X}(x',x'')$} (2);
      \draw [->,bend right=20] (0) to node[auto,swap,labelsize] {$\ecat{X}(x,x'')$} (2);
      \draw [2cell] (2,0.5) to node[auto,labelsize] {$M_{x,x',x''}$} (2,-0.5);
    \end{tikzpicture}    
    \]
    in $\bcat{B}$,
\end{itemize}
subject to the associativity and identity laws, generalizing the usual axioms for a category.

Since the isomorphism (\ref{eqn:cat_slice}) already forces us to consider enrichment over bicategories, it is natural to wonder whether there is a generalization of the isomorphism involving a bicategory $\bcat{B}$ in place of the monoidal category $\bcat{V}$. Indeed this turns out to be the case: for any bicategory $\bcat{B}$ and  $\bcat{B}$-category $\ecat{X}$, there is a bicategory $\bcat{B}/\ecat{X}$ with a canonical isomorphism of 2-categories $\enCat{(\bcat{B}/\ecat{X})}\cong \enCat{\bcat{B}}/\ecat{X}$.
Thus 2-categories of the form $\enCat{\bcat{B}}$ are closed under slicing, provided that we allow $\bcat{B}$ to range over bicategories.

The construction of $\bcat{B}/\ecat{X}$ is simple enough to carry out at this point; see also Remark~\ref{rmk:tricat-enrichment} for a more abstract point of view. We set $\ob(\bcat{B}/\ecat{X})=\ob(\ecat{X})$ and, for all $x,x'\in \ob(\bcat{B}/\ecat{X})$, the hom-category $(\bcat{B}/\ecat{X})(x,x')$ is the slice category 
$\bcat{B}(|x|,|x'|)/\ecat{X}(x,x')$. The identity 1-cell at $x$ is $j_x$, and the composite of 1-cells
$(s\colon S\to \ecat{X}(x,x'))\colon x\to x'$ and $(t\colon T\to \ecat{X}(x',x''))\colon x'\to x''$ is the pasting composite
\[
    \begin{tikzpicture}[baseline=-\the\dimexpr\fontdimen22\textfont2\relax ]
      \node(0) at (-0.5,0) {$|x|$};
      \node(1) at (2,0) {$|x'|$};
      \node(2) at (4.5,0) {$|x''|$.};
      
      \draw [->,bend left=40] (0) to node[auto,labelsize] {$S$} (1);
      \draw [->,bend left=40] (1) to node[auto,labelsize] {$T$} (2);
      \draw [->,bend left=0] (0) to node[auto,swap,labelsize] {$\ecat{X}(x,x')$} (1);
      \draw [->,bend left=0] (1) to node[auto,swap,labelsize] {$\ecat{X}(x',x'')$} (2);
      \draw [->,bend right=35] (0) to node[auto,swap,labelsize] {$\ecat{X}(x,x'')$} (2);
      \draw [2cell] (2,-0.1) to node[auto,labelsize] {$M_{x,x',x''}$} (2,-1.1);
      \draw [2cell] (0.75,0.8) to node[auto,labelsize] {$s$} (0.75,-0.2);
      \draw [2cell] (3.25,0.8) to node[auto,labelsize] {$t$} (3.25,-0.2);
    \end{tikzpicture}
\]
Of course, when both $\bcat{B}$ and $\ecat{X}$ have only one object, the construction of $\bcat{B}/\ecat{X}$ reduces to that of the slice of a monoidal category over a monoid. 

This observation allows one to view (enriched) functors as (enriched) categories, and suggests new perspectives even on notions which are not directly related to enrichment. 
For example, for any ($\Set$-)category $\ecat{X}$, there is a bicategory $\Set/\ecat{X}$ with an isomorphism $\enCat{(\Set/\ecat{X})}\cong \Cat/\ecat{X}$.
Thus we can view functors into $\ecat{X}$ as enriched categories (see Example~\ref{ex:Set/X} below and \cite{Garner-total} for a related construction), and we may potentially interpret properties of functors via enriched categorical terms. Indeed, we shall show that a functor $\ecat{Y}\to \ecat{X}$ is a Grothendieck fibration if and only if the corresponding $\Set/\ecat{X}$-category $\overline{\ecat{Y}}$ has powers by a certain class of 1-cells in $\Set/\ecat{X}$, as well as a $\bcat{B}$-enriched version of this result. 

The notation $\bcat{B}/\ecat{X}$ is justified by its characterization as the oplax limit of a 1-cell in a suitable 2-category. 
To explain this, recall that a $\bcat{B}$-category $\ecat{X}$ can be given equivalently as a lax functor $\ecat{X}\colon X_c\to \bcat{B}$, where $X_c$ is the chaotic category with the same set of objects as $\ecat{X}$.\footnote{Lax functors of this form were studied by B\'enabou \cite{Benabou-bicat} under the name {\em polyad}; for the connection with enriched categories see \cite{Street-cohomology}.} 
Thus we can view the $\bcat{B}$-category $\ecat{X}$ as a 1-cell in the 2-category $\Bicat$ of bicategories, lax functors and icons \cite{Lack_icons}. 
The bicategory $\bcat{B}/\ecat{X}$ is the oplax limit of this 1-cell in $\Bicat$:
\[
    \begin{tikzpicture}[baseline=-\the\dimexpr\fontdimen22\textfont2\relax ]
      \node(0) at (0,0) {$\bcat{B}/\ecat{X}$};
      \node(1) at (2,1) {$X_c$};
      \node(2) at (2,-1){$\bcat{B}$.};
      
      \draw [->] (0) to node[auto,labelsize] {} (1);
      \draw [->] (0) to node[auto,swap,labelsize] {} (2);
      \draw [->] (1) to node[auto,labelsize] {$\ecat{X}$} (2);
      \draw [2cell] (1.3,-0.7) to node[auto,swap,labelsize] {} (1.3,0.7);
    \end{tikzpicture}
\]
(Although $\Bicat$ is not complete, it does have oplax limits of 1-cells \cite{lax,LackShulman}.)
This generalizes the characterization of the slice monoidal category $\bcat{V}/M$ as the oplax limit of the monoid $M$ in $\bcat{V}$, regarded as a lax monoidal functor from the terminal monoidal category to $\bcat{V}$, in the 2-category of monoidal categories, lax monoidal functors and monoidal natural transformations.

In this paper, we study properties of the 2-functor $\Enr\colon \Bicat\to \twoCAT$ mapping each bicategory $\bcat{B}$ to the 2-category $\enCat{\bcat{B}}$, in order to understand the isomorphism $\enCat{(\bcat{B}/\ecat{X})}\cong \enCat{\bcat{B}}/\ecat{X}$ conceptually, as well as to establish further closure properties of 2-categories of the form $\enCat{\bcat{B}}$. 
To this end, it is useful to factorize $\Enr$ as 
\[
    \begin{tikzpicture}[baseline=-\the\dimexpr\fontdimen22\textfont2\relax ]
      \node(0) at (0,0) {$\Bicat$};
      \node(1) at (3,-1) {$\twoCAT/\Enr(\mathbf{1})$};
      \node(2) at (6,0){$\twoCAT$,};
      
      \draw [->] (0) to node[auto,swap,labelsize] {$\Enr_\mathbf{1}$} (1);
      \draw [->] (0) to node[auto,labelsize] {$\Enr$} (2);
      \draw [->] (1) to node[auto,swap,labelsize] {forgetful} (2);
    \end{tikzpicture}
\]
where $\mathbf{1}$ is the terminal bicategory. The 2-functor $\Enr_\mathbf{1}$ maps each bicategory $\bcat{B}$ to $\enCat{\bcat{B}}$ equipped with the 2-functor $\Enr(!)\colon\enCat{\bcat{B}}\to \Enr(\mathbf{1})$ induced from the unique lax functor $!\colon \bcat{B}\to \mathbf{1}$. The underlying category of $\Enr(\mathbf{1})$ is $\Set$, and $\Enr(!)$ can be regarded as $\ob(-)$, mapping each $\bcat{B}$-category $\ecat{X}$ to its set of objects $\ob(\ecat{X})$.
(Although $\Enr$ is usually denoted simply as $\enCat{(-)}$, we adopted the current notation in order to avoid the potentially misleading expression $\enCat{\mathbf{1}}$.)

In our main theorem (Theorem~\ref{thm:Enr-preserves-limits}), we show that $\Enr_\mathbf{1}\colon\Bicat\to\twoCAT/\Enr(\mathbf{1})$ preserves \emph{any} limit which happens to exist in $\Bicat$. This implies that $\Enr$ preserves any limit which happens to exist in $\Bicat$ \emph{and} is preserved by the forgetful 2-functor $\twoCAT/\Enr(\mathbf{1})\to \twoCAT$; the latter condition is satisfied whenever the limit in question is small enough to exist in $\twoCAT$ and is created by the forgetful 2-functor. 
In ordinary category theory, the limits created by the forgetful functors from slice categories are precisely the connected limits. 
In Section~\ref{sec:Cat-connected-limits} we generalize this to 2-categories (or in fact to $\bcat{V}$-categories where $\bcat{V}$ is any complete and cocomplete cartesian closed category), introducing the class of $\Cat$-connected limits with several characterizations.
Thus $\Enr\colon\Bicat\to\twoCAT$ preserves any $\Cat$-connected limit which happens to exist in $\Bicat$. This includes Eilenberg--Moore objects of comonads, for example.
Although oplax limits of 1-cells are not $\Cat$-connected, the isomorphism $\enCat{(\bcat{B}/\ecat{X})}\cong \enCat{\bcat{B}}/\ecat{X}$ is explained via the limit-preservation property of $\Enr$ and a 2-categorical argument in Section~\ref{sect:oplax}. 

Finally, in Section~\ref{sect:fibrations}, we investigate (internal) fibrations in the 2-category $\enCat{\bcat{B}}$ of $\bcat{B}$-categories. Specifically, we show that (assuming a mild local completeness condition on $\bcat{B}$) a $\bcat{B}$-functor $\ecat{Y}\to\ecat{X}$ is a fibration in $\enCat{\bcat{B}}$ if and only if the corresponding $\bcat{B}/\ecat{X}$-category $\overline{\ecat{Y}}$ admits certain powers.

We intend to revisit the results of this paper in the future, in the context of enrichment over pseudo double categories.

\section{The limit-preservation theorem}
\label{sec:limit-pres}
Size does not play a significant role in this paper; nonetheless we make a few comments here about the issues which arise and our approach to them. The typical {\em monoidal} categories over which one enriches, such as $\Set$, $\Cat$, or $\Ab$, have small hom-sets but are not themselves small. Thus the corresponding bicategories will not even have small hom-categories. We do still need some control of the size of these bicategories, and therefore fix
two Grothendieck universes $\cat{U}_0$ and $\cat{U}_1$ with $\cat{U}_0\in\cat{U}_1$. Sets, categories, etc.~in $\cat{U}_0$ and $\cat{U}_1$ are called \emph{small} and \emph{large} respectively.

Let $\Bicat$ be the 2-category of large bicategories, lax functors and icons \cite[Theorem~3.2]{Lack_icons}, and $\twoCAT$ be the 2-category of large 2-categories, 2-functors and 2-natural transformations.
We have a 2-functor $\Enr\colon\Bicat\to\twoCAT$ sending each bicategory $\bcat{B}$ to the 2-category $\enCat{\bcat{B}}$ of all small $\bcat{B}$-categories, $\bcat{B}$-functors and $\bcat{B}$-natural transformations.
It is the limit-preservation properties of this 2-functor $\Enr$ that is our main focus. The limits in question will be 2-limits weighted by 2-functors of the form $F\colon\bcat{D}\to \CAT$, where $\bcat{D}$ is a large 2-category and $\CAT$ is the 2-category of large categories.

The bicategory $\one$ with a single 2-cell is the terminal object of $\Bicat$, and hence $\Enr$ induces the 2-functor $\Enr_\mathbf{1}\colon \Bicat\to \twoCAT/\Enr(\one)$,
where $\twoCAT/\Enr(\one)$ denotes the (strict) slice 2-category of $\twoCAT$ over $\Enr(\one)\in\twoCAT$.
The 2-category $\Enr(\one)$ is the locally chaotic 2-category whose underlying category is $\Set$.
More precisely, the objects of $\Enr(\one)$ can be identified with the small sets, and for each pair of small sets $X$ and $Y$ we have $\Enr(\one)(X,Y)=\Set(X,Y)_c$, where $(-)_c$ appears in the string of adjunctions 
\begin{equation}
\label{eqn:adj-between-Set-and-Cat}
    \begin{tikzpicture}[baseline=-\the\dimexpr\fontdimen22\textfont2\relax ]
      \node(0) at (0,0) {$\SET$};
      \node(1) at (5,0) {$\CAT_0$.};
      
      \draw [<-,transform canvas={yshift=0},bend left=30] (0) to node[auto,labelsize] {$\pi_0$} (1);
      \draw [->,transform canvas={yshift=0},bend left=12] (0) to node[midway,fill=white,swap,labelsize] {$(-)_d$} (1); 
      \draw [<-,transform canvas={yshift=0},bend right=12] (0) to node[midway,fill=white,swap,labelsize] {$\ob$} (1);       
      \draw [->,transform canvas={yshift=0},bend right=30] (0) to node[auto,swap,labelsize] {$(-)_c$} (1); 
      \node [rotate=90] at (2.5,0.67) {$\vdash$}; 
      \node [rotate=90] at (2.5,0) {$\vdash$};             
      \node [rotate=90] at (2.5,-0.67) {$\vdash$}; 
    \end{tikzpicture}
\end{equation}
Here, $\SET$ and $\CAT_0$ denote the categories of large sets and of large categories respectively.
The (finite-product-preserving) functors in (\ref{eqn:adj-between-Set-and-Cat}) induce 2-adjunctions
\[
    \begin{tikzpicture}[baseline=-\the\dimexpr\fontdimen22\textfont2\relax ]
      \node(0) at (0,0) {$\CAT$};
      \node(1) at (5,0) {$\twoCAT$.};
      
      \draw [<-,transform canvas={yshift=0},bend left=30] (0) to node[auto,labelsize] {$(\pi_0)_\ast$} (1);
      \draw [->,transform canvas={yshift=0},bend left=12] (0) to node[midway,fill=white,swap,labelsize] {$(-)_{ld}$} (1); 
      \draw [<-,transform canvas={yshift=0},bend right=12] (0) to node[midway,fill=white,swap,labelsize] {$(-)_0$} (1);       
      \draw [->,transform canvas={yshift=0},bend right=30] (0) to node[auto,swap,labelsize] {$(-)_{lc}$} (1); 
      \node [rotate=90] at (2.5,0.67) {$\vdash$}; 
      \node [rotate=90] at (2.5,0) {$\vdash$};             
      \node [rotate=90] at (2.5,-0.67) {$\vdash$}; 
    \end{tikzpicture}
\]
Thus we shall write the 2-category $\Enr(\one)$ as $\Setlc$.

Explicitly, the 2-functor $\Enr_\one\colon\Bicat\to\twoCAT/\Setlc$ maps each bicategory $\bcat{B}$ to the 2-category $\enCat{\bcat{B}}$ equipped with the 2-functor $\ob(-)\colon\enCat{\bcat{B}}\to\Setlc$ which extracts the set of objects of a $\bcat{B}$-category.


\begin{theorem}
\label{thm:Enr-preserves-limits}
The 2-functor $\Enr_\one\colon\Bicat\to\twoCAT/\Setlc$ preserves all weighted limits which happen to exist in $\Bicat$.
\end{theorem}
\begin{proof}
We shall show the following.
\begin{enumerate}[label=(\Roman*)]
    \item[(a)] The set $\cat{G}$ of all objects of $\twoCAT/\Setlc$ of the form $(\twocell\to\Setlc)$, where $\twocell$ denotes the free 2-category on a single 2-cell, is a strong generator of the 2-category $\twoCAT/\Setlc$.
    \item[(b)] For each object $A\in\cat{G}$, the 2-functor $\twoCAT/\Setlc(A,\Enr_\one(-))\colon \Bicat\to\CAT$ is a 2-limit of representable 2-functors, and hence preserves all weighted limits which happen to exist in $\Bicat$.
\end{enumerate}
From these, the main claim follows. Indeed, let $\bcat{D}$ be a large 2-category, $F\colon \bcat{D}\to\CAT$ be a 2-functor (the weight) and $S\colon \bcat{D}\to \Bicat$ be a 2-functor such that the weighted limit $\{F,S\}$ exists in $\Bicat$. Then the weighted limit $\{F,\Enr_\one\circ S\}$ exists in $\twoCAT/\Setlc$, because $\twoCAT/\Setlc$ has all (large) weighted limits. We have a comparison 1-cell $M\colon \Enr_\one\{F,S\}\to\{F,\Enr_\one\circ S\}$ in $\twoCAT/\Setlc$. Now for each $A\in\cat{G}$, the functor
\[
\twoCAT/\Setlc(A,M)\colon\twoCAT/\Setlc(A,\Enr_\one\{F,S\})\to\twoCAT/\Setlc(A,\{F,\Enr_\one\circ S\})
\]
is an isomorphism by (b), from which we conclude that $M$ is an isomorphism by (a).

$\cat{G}$ is a strong generator of $\twoCAT/\Setlc$ because, given any 1-cell $T\colon (\bcat{X}\to\Setlc)\to(\bcat{Y}\to\Setlc)$, i.e., a 2-functor $T\colon \bcat{X}\to\bcat{Y}$ between 2-categories $\bcat{X}$ and $\bcat{Y}$ over $\Setlc$, the condition that $\twoCAT/\Setlc(A,T)$ be an isomorphism for all $A\in\cat{G}$ means that $T$ is bijective on 2-cells.

To show (b), observe that a 2-functor $\twocell\to\Setlc$ corresponds to a parallel pair of functions $f_0,f_1\colon X\to Y$. Such a 2-functor can be seen as an object of $\twoCAT/\Setlc$. Given $((f_0,f_1)\colon\twocell\to\Setlc)$ where $f_0,f_1\colon X\to Y$, first 
consider the category $\two\times X_c$ where $\two=\{0<1\}$ is the two-element chain.
We regard $\two\times X_c$ as a bicategory as well.
We have the projection functor $\pi\colon \two\times X_c\to X_c$ and the functor $[f_0,f_1]\colon \two\times X_c\to Y_c$ defined by $[f_0,f_1](i,x)=f_i(x)$;  
these can also be regarded as lax functors, i.e., morphisms in $\Bicat$.
The 2-functor 
\[
\twoCAT/\Setlc((f_0,f_1),\Enr_\one(-))\colon\Bicat\to\CAT
\] 
is the comma object (in $[\Bicat,\CAT]$) as in 
\[
    \begin{tikzpicture}[baseline=-\the\dimexpr\fontdimen22\textfont2\relax ]
      \node(01) at (0,1) {$\twoCAT/\Setlc((f_0,f_1),\Enr_\one(-))$};
      \node(00) at (0,-1) {$\Bicat(Y_c,-)$};
      \node(11) at (6,1) {$\Bicat(X_c,-)$};
      \node(10) at (6,-1) {$\Bicat(\two\times X_c,-)$.};
      
      \draw [->] (01) to node[auto,labelsize] {} (11);
      \draw [->] (11) to node[auto,labelsize] {$\Bicat(\pi,-)$} (10);
      \draw [->] (01) to node[auto,swap,labelsize] {} (00);
      \draw [->] (00) to node[auto,swap,labelsize] {$\Bicat([f_0,f_1],-)$} (10);
      \draw [2cell] (3.5,0.5) to (2.5,-0.5);
    \end{tikzpicture}
\]
Indeed, for any bicategory $\bcat{B}\in\Bicat$, an object of the comma category of the functors $\Bicat([f_0,f_1],\bcat{B})$ and $\Bicat(\pi,\bcat{B})$ consists of lax functors $\ecat{C}\colon X_c\to\bcat{B}$ and $\ecat{D}\colon Y_c\to\bcat{B}$ together with an icon
\[
    \begin{tikzpicture}[baseline=-\the\dimexpr\fontdimen22\textfont2\relax ]
      \node(01) at (0,1) {$\two\times X_c$};
      \node(00) at (0,-1) {$Y_c$};
      \node(11) at (2,1) {$X_c$};
      \node(10) at (2,-1) {$\bcat{B}$.};
      
      \draw [->] (01) to node[auto,labelsize] {$\pi$} (11);
      \draw [->] (11) to node[auto,labelsize] {$\ecat{C}$} (10);
      \draw [->] (01) to node[auto,swap,labelsize] {$[f_0,f_1]$} (00);
      \draw [->] (00) to node[auto,swap,labelsize] {$\ecat{D}$} (10);
      \draw [2cell] (1.5,0.5) to node[auto,swap,labelsize] {$\alpha$} (0.5,-0.5);
    \end{tikzpicture}
\]
This corresponds to $\bcat{B}$-categories $\ecat{C}$ and $\ecat{D}$ with $\ob(\ecat{C})=X$ and $\ob(\ecat{D})=Y$ such that $|x|_\ecat{C}=|f_i(x)|_\ecat{D}$ for all $x\in X$ and $i\in\{0,1\}$, together with a 2-cell $\alpha_{(i,x),(i',x')}\colon \ecat{C}(x,x')\to \ecat{D}(f_i(x),f_{i'}(x'))$ in $\bcat{B}$ for all $(i,x),(i',x')\in \two\times X_c$ with $i\leq i'$, satisfying some equations.
These latter data in turn correspond to $\bcat{B}$-functors $F_0\colon\ecat{C}\to\ecat{D}$ and $F_1\colon\ecat{C}\to\ecat{D}$ (with $\ob(F_i)=f_i$) together with a $\bcat{B}$-natural transformation $\alpha\colon F_0\to F_1$. 
(We record in Lemma~\ref{lem:B-nat-tr} below an observation which is useful for the verification.)

This gives a bijective correspondence on objects of $\twoCAT/\Setlc((f_0,f_1),\Enr_\mathbf{1}(\bcat{B}))$ and the comma category of $\Bicat([f_0,f_1],\bcat{B})$ and $\Bicat(\pi,\bcat{B})$, which routinely extends to an isomorphism of categories natural in $\bcat{B}$.
\end{proof}

\begin{lemma}
\label{lem:B-nat-tr}
Let $\bcat{B}$ be a bicategory, $\ecat{C},\ecat{D}$ be $\bcat{B}$-categories and $T,S\colon\ecat{C}\to\ecat{D}$ be $\bcat{B}$-functors. To give a $\bcat{B}$-natural transformation $\alpha\colon T\to S$, i.e., a family of 2-cells
\[
    \begin{tikzpicture}[baseline=-\the\dimexpr\fontdimen22\textfont2\relax ]
      \node(0) at (0,0) {$|x|$};
      \node(1) at (3,0) {$|x|$};
      
      \draw [->,bend left=20] (0) to node[auto,labelsize] {$1_{|x|}$} (1);
      \draw [->,bend right=20] (0) to node[auto,swap,labelsize] {$\ecat{D}(Tx,Sx)$} (1);
      \draw [2cell] (1.5,0.5) to node[auto,labelsize] {$\alpha_x$} (1.5,-0.5);
    \end{tikzpicture}    
\]
in $\bcat{B}$ for all $x\in\ecat{C}$, satisfying the naturality axiom saying that for all $x,x'\in\ecat{C}$,
\begin{equation}
\label{eqn:naturality}
\begin{tikzpicture}[baseline=-\the\dimexpr\fontdimen22\textfont2\relax ]
      \node(0) at (2,0) {$\ecat{C}(x,x')$};
      \node(10) at (2,1.5) {$1_{|x'|}.\ecat{C}(x,x')$};
      \node(11) at (2,-1.5) {$\ecat{C}(x,x') . 1_{|x|}$};
      \node(20) at (7,1.5) {$\ecat{D}(Tx',Sx') .\ecat{D}(Tx,Tx')$};
      \node(21) at (7,-1.5) {$\ecat{D}(Sx,Sx') .\ecat{D}(Tx,Sx)$};
      \node(3) at (7,0) {$\ecat{D}(Tx,Sx')$};
      
      \draw [->] (0) to node[auto, labelsize] {$\cong$} (10); 
      \draw [->] (10) to node[auto, labelsize] {$\alpha_{x'} . T_{x,x'}$} (20); 
      \draw [->] (20) to node[auto,labelsize] {$M^{\ecat{D}}_{Tx,Tx',Sx'}$} (3); 
      \draw [->] (0) to node[auto,swap,labelsize] {$\cong$} (11); 
      \draw [->] (11) to node[auto,swap,labelsize] {$S_{x,x'} . \alpha_x$} (21); 
      \draw [->] (21) to node[auto,swap,labelsize] {$M^\ecat{D}_{Tx,Sx,Sx'}$} (3); 
\end{tikzpicture}
\end{equation}
commutes, is equivalent to giving a family of 2-cells 
\[
    \begin{tikzpicture}[baseline=-\the\dimexpr\fontdimen22\textfont2\relax ]
      \node(0) at (0,0) {$|x|$};
      \node(1) at (3,0) {$|x'|$};
      
      \draw [->,bend left=20] (0) to node[auto,labelsize] {$\ecat{C}(x,x')$} (1);
      \draw [->,bend right=20] (0) to node[auto,swap,labelsize] {$\ecat{D}(Tx,Sx')$} (1);
      \draw [2cell] (1.5,0.5) to node[auto,labelsize] {$\alpha_{x,x'}$} (1.5,-0.5);
    \end{tikzpicture}    
\]
in $\bcat{B}$ for all $x,x'\in\ecat{C}$, such that for all $x,x',x''\in\ecat{C}$,
\[
    \begin{tikzpicture}[baseline=-\the\dimexpr\fontdimen22\textfont2\relax ]
      \node(01) at (0,0.75) {$\ecat{C}(x',x'') .\ecat{C}(x,x')$};
      \node(00) at (0,-0.75) {$\ecat{D}(Sx',Sx'') .\ecat{D}(Tx,Sx')$};
      \node(11) at (5,0.75) {$\ecat{C}(x,x'')$};
      \node(10) at (5,-0.75) {$\ecat{D}(Tx,Sx'')$};
      
      \draw [->] (01) to node[auto,labelsize] {$M^\ecat{C}_{x,x',x''}$} (11);
      \draw [->] (11) to node[auto,labelsize] {$\alpha_{x,x''}$} (10);
      \draw [->] (01) to node[auto,swap,labelsize] {$S_{x',x''} . \alpha_{x,x'}$} (00);
      \draw [->] (00) to node[auto,swap,labelsize] {$M^\ecat{D}_{Tx,Sx',Sx''}$} (10);
      
      \node(01a) at (0,-2.5) {$\ecat{C}(x',x'') .\ecat{C}(x,x')$};
      \node(00a) at (0,-4) {$\ecat{D}(Tx',Sx'') .\ecat{D}(Tx,Tx')$};
      \node(11a) at (5,-2.5) {$\ecat{C}(x,x'')$};
      \node(10a) at (5,-4) {$\ecat{D}(Tx,Sx'')$};
      
      \draw [->] (01a) to node[auto,labelsize] {$M^\ecat{C}_{x,x',x''}$} (11a);
      \draw [->] (11a) to node[auto,labelsize] {$\alpha_{x,x''}$} (10a);
      \draw [->] (01a) to node[auto,swap,labelsize] {$\alpha_{x',x''} . T_{x,x'}$} (00a);
      \draw [->] (00a) to node[auto,swap,labelsize] {$M^\ecat{D}_{Tx,Tx',Sx''}$} (10a);
    \end{tikzpicture}
    \qquad\text{and}
\]
commute; the correspondence is given by mapping $(\alpha_x)$ to $(\alpha_{x,x'})$ whose component at $(x,x')$ is the composite (\ref{eqn:naturality}).
\end{lemma}

As observed in \cite[Section~6.2]{Lack_icons}, the 2-category $\Bicat$ can be seen as the 2-category of strict algebras, lax morphisms, and algebra 2-cells for a 2-monad $T$ on a certain locally presentable 2-category of $\CAT$-enriched graphs, and so by \cite{lax} has oplax limits, Eilenberg--Moore objects of comonads, and limits of diagrams containing only strict morphisms; this last class includes in particular products and powers. It also has various other sorts of limits where certain parts of the diagram are required to be pseudofunctors. For a more precise characterization see \cite{LackShulman}. 

The case of oplax limits of 1-cells is our motivating example, and is formalized in Section~\ref{sect:oplax}, specifically in Theorem~\ref{thm:enrich-over-slice}.
The case of Eilenberg--Moore objects of comonads is treated in Example~\ref{ex:comonads}. 
As a final example, we consider products. In this case, Theorem~\ref{thm:Enr-preserves-limits} says that, for bicategories $\bcat{B}$ and $\bcat{C}$, the diagram
\[
    \begin{tikzpicture}[baseline=-\the\dimexpr\fontdimen22\textfont2\relax ]
      \node(01) at (0,1) {$\enCat{(\bcat{B}\times\bcat{C})}$};
      \node(11) at (3,1) {$\enCat{\bcat{C}}$};
      \node(00) at (0,-1){$\enCat{\bcat{B}}$};
      \node(10) at (3,-1){$\Setlc$};
      \draw [->] (01) to node[auto,labelsize] {} (11);
      \draw [->] (11) to node[auto,labelsize] {$\ob$} (10);
      \draw [->] (01) to node[auto,swap,labelsize] {} (00);
      \draw [->] (00) to node[auto,swap, labelsize] {$\ob$} (10);
    \end{tikzpicture}
\]
is a pullback of 2-categories. In particular, to give a  $\bcat{B}\times\bcat{C}$-category is equivalent to giving a $\bcat{B}$-category and a $\bcat{C}$-category with the same set of objects.

\begin{remark}
It is possible to remove any size-related conditions on the notion of weighted limit in Theorem~\ref{thm:Enr-preserves-limits}. That is, for any (possibly larger than ``large'') 2-category $\bcat{D}$ and a weight $F\colon \bcat{D}\to \CAT'$, where $\CAT'$ is a 2-category of categories in a universe containing $\mathcal{U}_1$, $\Enr_\mathbf{1}$ preserves all $F$-weighted limits which happen to exist in $\Bicat$.
Indeed, let $S\colon \bcat{D}\to \Bicat$ be a 2-functor such that $\{F,S\}$ exists in $\Bicat$. Then, although a priori we do not know if $\{F,\Enr_\one\circ S\}$ exists in $\twoCAT/\Setlc$ or not, we can certainly consider a large enough variant $\twoCAT'/\Setlc$ in which it does. Then by the above discussion we have $\Enr_\one\{F,S\}\cong\{F,\Enr_\one\circ S\}$ in $\twoCAT'/\Setlc$. Since the fully faithful 2-functor $\twoCAT/\Setlc\to\twoCAT'/\Setlc$ reflects limits, and $\Enr_\one$ does land in $\twoCAT/\Setlc$, we see that the limit $\{F,\Enr_\one\circ S\}$ actually exists in $\twoCAT/\Setlc$.
\end{remark}

\section{Weighted limits created by forgetful 2-functors $\bcat{K}/A\to\bcat{K}$}
\label{sec:Cat-connected-limits}
Theorem~\ref{thm:Enr-preserves-limits} implies that the 2-functor $\Enr\colon\Bicat\to \twoCAT$ preserves all weighted limits preserved by the forgetful 2-functor $\twoCAT/\Setlc\to\twoCAT$. 
We now investigate these. 

A large part of this section (until the end of Example~\ref{ex:products-etc}) is devoted to the study of this class of limits, which we shall call $\Cat$-connected. Since this notion does not require two separate universes, and since it may be of interest in other contexts, we work with a single universe $\mathcal{U}$, whose elements we call \emph{small} sets. (We temporarily ignore $\mathcal{U}_0$ introduced at the beginning of Section~\ref{sec:limit-pres}.)
When we later return to the study of $\Bicat$ and $\twoCAT/\Setlc$, we apply our results in the case $\mathcal{U}=\mathcal{U}_1$, and so speak of $\CAT$-connected limits.


In the literature there are (at least) two definitions of creation of limit. 
Given 2-functors $F\colon \bcat{D}\to\Cat$, $S\colon \bcat{D}\to\bcat{A}$, and $G\colon \bcat{A}\to\bcat{B}$, the phrase ``$G$ creates the $F$-weighted limit of $S$'' could mean either of the following.
\begin{itemize}
    \item For any $F$-weighted limit 
    $(\{F,GS\},\mu\colon F\to\bcat{B}(\{F,GS\},GS-))$ of $GS$, there exists a unique $F$-cylinder $(L,\nu\colon F\to\bcat{A}(L,S-))$ over $S$ in $\bcat{A}$ such that $\{F,GS\}=GL$ and $\mu=G_{L,S-}\circ \nu$ hold. Moreover, $(L,\nu)$ is an $F$-weighted limit of $S$.
    \item For any $F$-weighted limit 
    $(\{F,GS\},\mu)$ of $GS$, there exists an $F$-cylinder $(L,\nu)$ over $S$ in $\bcat{A}$ such that the mediating 1-cell $GL\to \{F,GS\}$ is an isomorphism. Moreover, such an $F$-cylinder $(L,\nu)$ is always an $F$-weighted limit of $S$.
\end{itemize}
These two conditions are equivalent when $G$ is the forgetful 2-functor $\bcat{K}/A\to \bcat{K}$ from a slice 2-category, since such 2-functors reflect identities and lift invertible 1-cells.

In the following, $1$ and $\mathbf{1}$ denote the terminal category and the terminal 2-category, respectively.

\begin{theorem}
\label{thm:strongly-connected-limits}
Let $\bcat{D}$ be a small 2-category and $F\colon \bcat{D}\to\Cat$ be a 2-functor. Then the following are equivalent.
\begin{enumerate}[label=\emph{(\arabic*)}]
    \item All $F$-weighted limits are created by the forgetful 2-functor $\bcat{K}/A\to\bcat{K}$ for any locally small 2-category $\bcat{K}$ and $A\in\bcat{K}$.
    \item All $F$-weighted limits commute with copowers in $\Cat$. In other words, $F$-weighted limits are preserved by the 2-functor $X\times  (-)\colon\Cat\to\Cat$ for any $ X\in\Cat$.
    \item The $F$-weighted limit of the unique 2-functor $\bcat{D}\to \mathbf{1}$ is preserved by any 2-functor $ \one\to \Cat$: that is, $ X\cong [\bcat{D},\Cat](F,\Delta X)$ for any $ X\in\Cat$.
    \item The $F$-weighted limit of the unique 2-functor $\bcat{D}\to \mathbf{1}$ is preserved by any 2-functor $\one\to\bcat{K}$: that is, $A\cong \{F,\Delta A\}$ for any locally small 2-category $\bcat{K}$ and  $A\in\bcat{K}$.
    \item $F\ast(-)\colon [\bcat{D}^\op,\Cat]\to\Cat$ preserves the terminal object. In other words, the $F$-weighted colimit of ${\Delta}1\colon \bcat{D}^\op\to\Cat$ is the terminal category: $F\ast{\Delta} 1 \cong 1$.
    \item The (conical) colimit of $F$ is the terminal category: $\Delta 1\ast F \cong 1$.
\end{enumerate}
\end{theorem}
\begin{proof}
{[(1)$\implies$(2)]}
For any $ X\in\Cat$, copowers by $X$ are given by $X\times (-)\colon \Cat\to\Cat$, which is the composite of the right adjoint 2-functor $X\times  (-)\colon \Cat\to\Cat/ X$ and the forgetful 2-functor $\Cat/ X\to\Cat$.

{[(2)$\implies$(3)]}
Note that we have $1\cong \{F,\Delta 1\}$ in $\Cat$. Since $X\times  (-)\colon\Cat\to\Cat$ preserves the $F$-weighted limit $\{F,\Delta 1\}$, we have $ X\cong \{F,\Delta  X\}$. 

{[(3)$\implies$(4)]}
For any $B\in \bcat{K}$ we have $\bcat{K}(B,A)\cong [\bcat{D},\Cat](F,\Delta\bcat{K}(B,A))$. This shows that $A\in\bcat{K}$ is the weighted limit $\{F,\Delta A\}$.

{[(4)$\implies$(1)]}
Let $T\colon\bcat{D}\to\bcat{K}/A$ be a 2-functor, with the corresponding oplax cone
\[
    \begin{tikzpicture}[baseline=-\the\dimexpr\fontdimen22\textfont2\relax ]
      \node(0) at (0,0) {$\bcat{D}$};
      \node(1) at (2,1) {$\one$};
      \node(2) at (2,-1){$\bcat{K}$.};
      
      \draw [->] (0) to node[auto,labelsize] {} (1);
      \draw [->] (0) to node[auto,swap,labelsize] {$S$} (2);
      \draw [->] (1) to node[auto,labelsize] {$A$} (2);
      \draw [2cell] (1.3,-0.7) to node[auto,swap,labelsize] {$\gamma$} (1.3,0.7);
    \end{tikzpicture}
\]
In particular, $S$ is the composite of $T$ and the forgetful 2-functor $\bcat{K}/A\to\bcat{K}$. Suppose that the weighted limit $\{F,S\}$ exists in $\bcat{K}$. We have a 1-cell $\{F,\gamma\}\colon \{F,S\}\to\{F,\Delta A\}\cong A$ in $\bcat{K}$. We claim that the object $(\{F,\gamma\}\colon\{F,S\}\to A)\in\bcat{K}/A$ is the limit $\{F,T\}$ in $\bcat{K}/A$.
For any $(p\colon B\to A)\in \bcat{K}/A$, the hom category $(\bcat{K}/A)((B,p),(\{F,S\},\{F,\gamma\}))$ is given by the equalizer
\[
    \begin{tikzpicture}[baseline=-\the\dimexpr\fontdimen22\textfont2\relax ]
      \node(0) at (-2,0) {$(\bcat{K}/A)((B,p),(\{F,S\},\{F,\gamma\}))$};
      \node(1) at (3.5,0) {$\bcat{K}(B,\{F,S\})$};
      \node(2) at (7.5,0) {$\bcat{K}(B,A)$,};      

      \draw [->] (0) to node[auto,labelsize] {} (1);
      \draw [->, transform canvas={yshift=3}] (1) to node[auto,labelsize] {$\bcat{K}(B,\{F,\gamma\})$} (2);
      \draw [->, transform canvas={yshift=-3}] (1) to node[auto,swap,labelsize] {$\Delta p$} (2);
    \end{tikzpicture}
\]
which is easily seen to be canonically isomorphic to $[\bcat{D},\Cat](F,(\bcat{K}/A)((B,p),T-))$.

[(4)$\implies$(5)] Applying (4) to $1\colon \one\to\Cat^\op$, we obtain $F\ast {\Delta} 1\cong 1$ in $\Cat$.

[(5)$\implies$(3)] For any $ X\in\Cat$, we have 
\[
X\cong [1, X]\cong [F\ast {\Delta} 1,  X]\cong [\bcat{D},\Cat](F, [{\Delta}1(-), X])\cong [\bcat{D},\Cat](F,{\Delta} X).
\]

[(5)$\iff$(6)] By $F\ast{\Delta}1 \cong{\Delta}1\ast F$.
\end{proof}

A 2-functor $F\colon\bcat{D}\to \Cat$ is called \emph{$\Cat$-connected} if $F$ satisfies the equivalent conditions of Theorem~\ref{thm:strongly-connected-limits}. Similarly, a weighted limit is \emph{$\Cat$-connected} if its weight is so.
Note that $F\colon\bcat{D}\to\Cat$ is connected (in the sense that $[\bcat{D},\Cat](F,-)\colon[\bcat{D},\Cat]\to\Cat$ preserves small coproducts) if and only if $[\bcat{D},\Cat](F,\Delta  X)\cong  X$ for any small discrete category $X$, or equivalently just for $X=1+1$; on the other hand it is $\Cat$-connected if this holds for all small categories $X$.

\begin{remark}
Theorem~\ref{thm:strongly-connected-limits} can be proved more generally for categories enriched over a complete and cocomplete cartesian closed category $\bcat{V}$ in place of $\Cat$, indeed the proof carries over essentially word-for-word upon replacing each instance of $\Cat$ by $\bcat{V}$.
\end{remark}

We now give a few simple results about $\Cat$-connected weights in order to clarify the scope of the notion.

\begin{proposition}
If $\bcat{D}$ has a terminal object, then $F\colon \bcat{D}\to\Cat$ is $\Cat$-connected if and only if $F$ preserves the terminal object. 
\end{proposition}

\begin{proof}
If $\bcat{D}$ has a terminal object $1$ then the colimit of $F$ is $F(1)$.
\end{proof}

\begin{proposition}\label{prop:strongly-connected-ordinary}
Let $\bcat{C}$ be a small ordinary category, and $G\colon\bcat{C}\to\Set$ a functor. This determines a 2-functor $G_d\colon\bcat{C}_{ld}\to\Cat$, where now $\bcat{C}_{ld}$ is regarded as a locally discrete 2-category. This $G_d$ sends an object $C$ to the discrete category $G(C)_d$ with object-set $G(C)$.
Then $G_d$ is $\Cat$-connected if and only if the corresponding $G$ is connected.
\end{proposition}

\begin{proof}
Since the functor $(-)_d\colon\Set\to\Cat_0$ preserves colimits, $\colim(G_d)=\colim(G)_d$. 
\end{proof}

\begin{proposition}\label{prop:Cat-connected-D0}
$\Delta 1\colon\bcat{D}\to\Cat$ is $\Cat$-connected if and only if $\bcat{D}_0$ is connected.
\end{proposition}

\begin{proof}
The colimit of $\Delta 1\colon\bcat{D}\to\Cat$ is 
the discrete category corresponding to the set of connected components of $\bcat{D}_0$.
\end{proof}

\begin{example}
Equifiers are $\Cat$-connected: here it is easiest to verify directly that equifiers in $\Cat$ commute with copowers. Similarly, one verifies that Eilenberg--Moore objects of monads and of comonads are $\Cat$-connected.
Equalizers and pullbacks are $\Cat$-connected by Proposition~\ref{prop:Cat-connected-D0}. 
\end{example}

\begin{example}
\label{ex:products-etc}
Non-trivial products are not $\Cat$-connected: they are not even connected. Powers by a category $X$ are limits weighted by $X\colon \one\to\Cat$; since the colimit of such a weight is just $X$, powers by $X$ are $\Cat$-connected if and only if $X=1$. Inserters, comma objects and oplax limits of 1-cells are not $\Cat$-connected: in particular they are not preserved by the 2-functor $\mathbb{N}\colon\mathbf{1}\to\Cat$ which picks out the additive monoid $\mathbb{N}$ of natural numbers. More generally, inserters are not preserved by $X\colon \mathbf{1}\to\Cat$ if $X$ has a non-identity endomorphism, while comma objects and oplax limits of 1-cells are not preserved by $X\colon \mathbf{1}\to\Cat$ unless $X$ is discrete. 
\end{example}

As anticipated at the beginning of the section, we now take $\mathcal{U}$ to be $\mathcal{U}_1$, and use the resulting notion of $\CAT$-connected limit, involving a large 2-category $\bcat{D}$ and a 2-functor $F\colon \bcat{D}\to\CAT$ as weight.
Since the inclusion $\Cat\to\CAT$ preserves small limits and small colimits, $\Cat$-connected limits are also $\CAT$-connected.
As an immediate consequence of Theorems~\ref{thm:Enr-preserves-limits} and~\ref{thm:strongly-connected-limits}, we have:

\begin{corollary}\label{cor:Enr-Cat-connected}
The 2-functor $\Enr\colon\Bicat\to\twoCAT$ preserves all $\CAT$-connected limits which happen to exist in $\Bicat$.
\end{corollary}

\begin{example}\label{ex:comonads}
Eilenberg--Moore objects of comonads are $\Cat$-connected (as well as $\CAT$-connected), and exist in $\Bicat$ by the results of \cite{lax, LackShulman}, thus they are preserved by $\Enr$. In more detail, a comonad $G$ in $\Bicat$ on a bicategory $\bcat{B}$ consists of a comonad $G=G_{a,b}$ on each hom-category $\bcat{B}(a,b)$, together with 2-cells $G_2\colon Gg.Gf\to G(gf)$ for all $f\colon a\to b$ and $g\colon b\to c$, and 2-cells $G_0\colon 1_{Ga}\to G1_a$ for all objects $a$, subject to various conditions, which say that the $G_{a,b}$, the $G_2$ and the $G_0$ can be assembled into an identity-on-objects lax functor $\bcat{B}\to\bcat{B}$, in such a way that the counits and comultiplications for the comonads become icons. The Eilenberg--Moore object $\bcat{B}^G$ is the bicategory with the same objects as $\bcat{B}$, and with hom-category $\bcat{B}^G(a,b)$ given by the Eilenberg--Moore category $\bcat{B}(a,b)^{G_{a,b}}$ of $G_{a,b}$. Corollary~\ref{cor:Enr-Cat-connected} then says that $\enCat{\bcat{B}^G}$ is the Eilenberg--Moore 2-category for the induced 2-comonad on $\enCat{\bcat{B}}$.
\end{example}

\section{Oplax limits and fibrations}\label{sect:oplax}

A 1-cell $f\colon A\to B$ in a 2-category $\bcat{K}$ is called a {\em fibration}, when $\bcat{K}(C,f)\colon\bcat{K}(C,A)\to\bcat{K}(C,B)$ is a Grothendieck fibration for each $C\in\bcat{K}$, and 
\[
    \begin{tikzpicture}[baseline=-\the\dimexpr\fontdimen22\textfont2\relax ]
      \node(01) at (0,1) {$\bcat{K}(C,A)$};
      \node(11) at (3,1) {$\bcat{K}(D,A)$};
      \node(00) at (0,-1){$\bcat{K}(C,B)$};
      \node(10) at (3,-1){$\bcat{K}(D,B)$};
      \draw [->] (01) to node[auto,labelsize] {$\bcat{K}(c,A)$} (11);
      \draw [->] (11) to node[auto,labelsize] {$\bcat{K}(D,f)$} (10);
      \draw [->] (01) to node[auto,swap,labelsize] {$\bcat{K}(C,f)$} (00);
      \draw [->] (00) to node[auto,swap, labelsize] {$\bcat{K}(c,B)$} (10);
    \end{tikzpicture}
\]
is a morphism of fibrations for each $c\colon D\to C$ in $\bcat{K}$, in the sense that $\bcat{K}(c,A)$ sends cartesian morphisms (with respect to $\bcat{K}(C,f)$) to cartesian morphisms (with respect to $\bcat{K}(D,f)$).
If $q\colon F\to B$ and $p\colon E\to B$ are fibrations in $\bcat{K}$ with the common codomain $B$, then a 1-cell $r\colon (F,q)\to (E,p)$ in $\bcat{K}/B$ is a \emph{morphism of fibrations} if for each $C\in \bcat{K}$, $\bcat{K}(C,r)$ is a morphism of fibrations, i.e., preserves cartesian morphisms.

As explained by Street \cite{Street-YonedaFibrations}, 
these notions can be reformulated if the 2-category $\bcat{K}$ has oplax limits of 1-cells, as we shall henceforth suppose. Recall that the oplax limit of a 1-cell $f\colon A\to B$ in $\bcat{K}$ is the universal diagram
\[
    \begin{tikzpicture}[baseline=-\the\dimexpr\fontdimen22\textfont2\relax ]
      \node(0) at (0,0) {$B/f$};
      \node(1) at (2,1) {$A$};
      \node(2) at (2,-1){$B$};
      \draw [->] (0) to node[auto,labelsize] {$u_f$} (1);
      \draw [->] (0) to node[auto,swap,labelsize] {$v_f$} (2);
      \draw [->] (1) to node[auto,labelsize] {$f$} (2);
      \draw [2cell] (1.3,-0.7) to node[auto,swap,labelsize] {$\lambda_f$} (1.3,0.7);
    \end{tikzpicture}
\]
wherein we often drop the subscripts $f$ unless multiple oplax limits are being used. 

If $\bcat{K}=\Cat$, then these oplax limits are comma categories, as the notation suggests. On the other hand, we have:

\begin{example}\label{ex:Xc-Cat}
Let $X$ be a small set, seen as a chaotic bicategory $X_c$ (that is, $(X_c)_{ld}$ or equivalently $(X_c)_{lc}$). To give an $X_c$-enriched category is just to give a set of objects with a map into $X$. Similar calculations involving $X_c$-enriched functors and natural transformations show that the diagram
\[
    \begin{tikzpicture}[baseline=-\the\dimexpr\fontdimen22\textfont2\relax ]
      \node(0) at (0,0) {$\enCat{X_c}$};
      \node(1) at (2,1) {$\mathbf{1}$};
      \node(2) at (2,-1){$\Setlc$};
      
      \draw [->] (0) to node[auto,labelsize] {} (1);
      \draw [->] (0) to node[auto,swap,labelsize] {$\ob$} (2);
      \draw [->] (1) to node[auto,labelsize] {$X$} (2);
      \draw [2cell] (1.3,-0.7) to node[auto,swap,labelsize] {} (1.3,0.7);
    \end{tikzpicture}
\]
is an oplax limit in $\twoCAT$; in other words, the 2-category $\enCat{X_c}$ is isomorphic to the slice 2-category $\Setlc/X$; this in turn is isomorphic to $(\Set/X)_{lc}$.
\end{example}
The fibrations in $\bcat{K}$ with codomain $B$ can be understood in terms of a 2-monad $T_B$ on $\bcat{K}/B$ whose underlying 2-functor maps $f\colon A\to B$ to $ v_f\colon B/f\to B$;  the component at $f\colon A\to B$ of its unit is the unique map $d=d_f\colon A\to B/f$ with $ud=1_A$, $vd=f$, and $\lambda d$ equal to the identity 2-cell on $f$.  This 2-monad is {\em colax idempotent} (has the dual of the ``Kock--Z\"oberlein property"), and so an object $f\colon A\to B$ of $\bcat{K}/B$ admits the structure of a pseudo $T_B$-algebra if and only if $d\colon (A,f)\to (B/f,v_f)$ has a right adjoint in $\bcat{K}/B$; and this in turn is the case if and only if $f$ is a fibration. See for example \cite[Proposition~3(a)]{Street-YonedaFibrations} and \cite[Theorem 2.7]{Weber-Yoneda}.

Also, if $q\colon F\to B$ and $p\colon E\to B$ are fibrations in $\bcat{K}$, then a 1-cell $r\colon (F,q)\to (E,p)$ in $\bcat{K}/B$ admits the structure of a (pseudo) morphism of pseudo $T_B$-algebras if and only if the mate of the identity 2-cell
 \[
    \begin{tikzpicture}[baseline=-\the\dimexpr\fontdimen22\textfont2\relax ]
      \node(01) at (0,-1) {$(B/q,v_q)$};
      \node(00) at (0,1) {$(F,q)$};
      \node(11) at (3,-1) {$(B/p,v_p)$};
      \node(10) at (3,1) {$(E,p)$};
      
      \draw [->] (01) to node[auto,labelsize,swap] {$T_Br$} (11);
      \draw [->] (10) to node[auto,labelsize,right] {$d_p$} (11);
      \draw [->] (00) to node[auto,swap,labelsize,left] {$d_q$} (01);
      \draw [->] (00) to node[auto,labelsize] {$r$} (10);
    \end{tikzpicture}
\]
is invertible; and this in turn is the case if and only if $r$ is a morphism of fibrations.
 
 Likewise, the {\em strict} $T_B$-algebras are the {\em split} fibrations in $\bcat{K}$: those $f\colon A\to B$ for which each $\bcat{K}(C,f)\colon \bcat{K}(C,A)\to\bcat{K}(C,B)$ is a split fibration, and each $\bcat{K}(c,A)\colon$ $\bcat{K}(C,A)\to\bcat{K}(D,A)$ preserves the {\em chosen} cartesian lifts. 
 
 In particular, $v\colon B/f\to B$ is a split fibration for any $f\colon A\to B$, and $d$ exhibits $v\colon B/f\to B$ as the free (split) fibration on $f$. Thus if $p\colon E\to B$ is a fibration, and $g\colon A\to E$ defines a morphism $(A,f)\to (E,p)$ in $\bcat{K}/B$, there is an essentially unique morphism of fibrations $r\colon (B/f,v)\to (E,p)$ extending $g$. 

\begin{proposition}
\label{prop:Enr-fib}
    The 2-functor $\Enr_\one\colon \Bicat\to \twoCAT/\Setlc$ factors through the locally full sub-2-category of $\twoCAT/\Setlc$ having
    \begin{itemize}
        \item the fibrations in $\twoCAT$ to $\Setlc$ as objects, and 
        \item the fibration morphisms as 1-cells.
    \end{itemize}
\end{proposition}
\begin{proof}
    First we describe fibrations in $\twoCAT$ explicitly. Given a 2-functor $F\colon \bcat{Y}\to\bcat{X}$ between 2-categories, a 1-cell $h\colon y'\to y$ in $\bcat{Y}$ is called \emph{cartesian} (with respect to $F$) if 
    \[
    \begin{tikzpicture}[baseline=-\the\dimexpr\fontdimen22\textfont2\relax ]
      \node(01) at (0,1) {$\bcat{Y}(z,y')$};
      \node(11) at (4,1) {$\bcat{Y}(z,y)$};
      \node(00) at (0,-1){$\bcat{X}(Fz,Fy')$};
      \node(10) at (4,-1){$\bcat{X}(Fz,Fy)$};

      \draw [->] (01) to node[auto,labelsize] {$\bcat{Y}(z,h)$} (11);
      \draw [->] (11) to node[auto,labelsize] {$F_{z,y}$} (10);
      \draw [->] (01) to node[auto,swap,labelsize] {$F_{z,y'}$} (00);
      \draw [->] (00) to node[auto,swap, labelsize] {$\bcat{X}(Fz,Fh)$} (10);
    \end{tikzpicture}
\]
is a pullback in $\CAT$ for each $z\in\bcat{Y}$. 
Then $F$ is a fibration if and only if, for each object $y\in\bcat{Y}$ and each 1-cell $g\colon x\to Fy$ in $\bcat{X}$, there is a cartesian morphism $\overline{g}\colon g^\ast y\to y$ in $\bcat{Y}$ with $F\overline{g}=g$; such a  $\overline{g}$ is called a \emph{cartesian lifting} of $g$ to $y$.
Moreover, given fibrations $F\colon \bcat{Y}\to \bcat{X}$ and $G\colon\bcat{Z}\to\bcat{X}$ over $\bcat{X}$, a 2-functor $H\colon \bcat{Y}\to\bcat{Z}$ satisfying $F=G\circ H$ is a morphism of fibrations if and only if $H$ preserves cartesian 1-cells.
(This is a special case of Proposition~\ref{prop:fibration-concrete} below, whose proof does not depend on the current proposition.)

For any $\bcat{B}\in\Bicat$, a $\bcat{B}$-functor $S\colon \ecat{Y}'\to\ecat{Y}$ is called \emph{fully faithful} when the 2-cell $S_{y,y'}\colon \ecat{Y}(y,y')\to \ecat{Y'}(Sy,Sy')$ in $\bcat{B}$ is invertible for all $y,y'\in \ecat{Y}$. It is easy to see that a $\bcat{B}$-functor is cartesian with respect to $\ob(-)\colon\enCat{\bcat{B}}\to\Setlc$ if it is fully faithful, and indeed by essential uniqueness of cartesian lifts the reverse implication also holds.
The claim follows at once.
\end{proof}

\begin{remark}
    In the above proposition, we used fibrations in the 2-category $\twoCAT$, called {\em 2-categorical fibrations} in \cite[I.2.9]{Gray-FCT}. These were also called \emph{2-fibrations} in \cite{Gray-FCT}, but for the purposes of this remark we shall save that name for the more restrictive notion studied by Hermida \cite{Hermida}; see also \cite{Bakovic,Buckley}. In general, $\ob(-)\colon \enCat{\bcat{B}}\to \Setlc$ is not a 2-fibration in the sense of \cite{Hermida}. 
    Indeed, a 2-fibration is a 2-functor which among other things is locally a fibration, but the forgetful functor $\enCat{\bcat{B}}(\ecat{X},\ecat{Y})\to \Setlc(\ob(\ecat{X}),\ob(\ecat{Y}))$ induced by $\ob(-)$ is rarely a fibration of categories.
\end{remark}

In general, oplax limits of 1-cells are not preserved by the projection $\bcat{K}/B\to\bcat{K}$, but to some extent fibrations can be used to remedy this, as the following result shows. 

\begin{proposition}\label{prop:oplax-limit-fibration}
Let $p\colon A\to B$ be a fibration in $\bcat{K}$, and consider a morphism $g$ in $\bcat{K}/B$ into $p$, and the (essentially unique) induced morphism $r$ of fibrations, as below
\[
    \begin{tikzpicture}[baseline=-\the\dimexpr\fontdimen22\textfont2\relax ]
      \node(1) at (2,1) {$C$};
      \node(2) at (2,-1){$A$};
      \node(3) at (4,0){$B$};
      
      \draw [->] (1) to node[auto,swap,labelsize] {$g$} (2);
      \draw [->] (1) to node[auto,labelsize] {$pg$} (3);
      \draw [->] (2) to node[auto,swap,labelsize] {$p$} (3);
    \end{tikzpicture}
    \qquad\qquad
    \begin{tikzpicture}[baseline=-\the\dimexpr\fontdimen22\textfont2\relax ]
      \node(1) at (2,1) {$B/pg$};
      \node(2) at (2,-1){$A$};
      \node(3) at (4,0){$B$.};
      
      \draw [->] (1) to node[auto,swap,labelsize] {$r$} (2);
      \draw [->] (1) to node[auto,labelsize] {$v_{pg}$} (3);
      \draw [->] (2) to node[auto,swap,labelsize] {$p$} (3);
    \end{tikzpicture}
\]
Then the oplax limit of $g$ in $\bcat{K}$ is the oplax limit of $r$ in $\bcat{K}/B$.
\end{proposition}
\begin{proof}



As usual we write $A/g$ for the oplax limit of $g$ in $\bcat{K}$. We also write $(A,p)/r$ for the oplax limit of $r$ in $\bcat{K}/B$.

A morphism $D\to A/g$ consists of morphisms $a\colon D\to A$, $c\colon D\to C$, and a 2-cell $\alpha\colon a\to gc$. 

A morphism $D\to B/pg$ consists of morphisms $b\colon D\to B$, $c\colon D\to C$, and a 2-cell $\beta\colon b\to pgc$, and composing with $r$ gives the domain of the cartesian lifting $\overline{\beta}\colon \beta^*gc\to gc$ of $\beta$. A morphism $(D,b)\to(A,p)/r$ in $\bcat{K}/B$ consists of $(b,c,\beta)\colon D\to B/pg$, $a\colon D\to A$, and a 2-cell $\alpha'\colon a\to \beta^*gc$ with $p\alpha'$ equal to the identity on $pa=b$. But by the fibration property of $p$, to give such an $\alpha'$ is equivalently to give $\alpha\colon a\to gc$ with $p\alpha=\beta$.


This shows that the one-dimensional aspect of the universal properties of $A/g$ and $(A,p)/r$ agree, and similarly the two-dimensional aspects also agree. 
\end{proof}

We can use this to prove the following key result, already stated in the introduction.

\begin{theorem}\label{thm:enrich-over-slice}
Let $\bcat{B}$ be a bicategory and $\ecat{X}$ a $\bcat{B}$-category. Then the slice 2-category  $\enCat{\bcat{B}}/\ecat{X}$ is isomorphic to $\enCat{(\bcat{B}/\ecat{X})}$ for a bicategory $\bcat{B}/\ecat{X}$. 
\end{theorem}

\begin{proof}
If we regard $\ecat{X}$ as a lax functor $\ecat{X}\colon X_c\to \bcat{B}$, where $X=\ob(\ecat{X})$, we may take its oplax limit 
\[
    \begin{tikzpicture}[baseline=-\the\dimexpr\fontdimen22\textfont2\relax ]
      \node(0) at (0,0) {$\bcat{B}/\ecat{X}$};
      \node(1) at (2,1) {$X_c$};
      \node(2) at (2,-1){$\bcat{B}$};
      
      \draw [->] (0) to node[auto,labelsize] {} (1);
      \draw [->] (0) to node[auto,swap,labelsize] {} (2);
      \draw [->] (1) to node[auto,labelsize] {$\ecat{X}$} (2);
      \draw [2cell] (1.3,-0.7) to node[auto,swap,labelsize] {} (1.3,0.7);
    \end{tikzpicture}
\]
in $\Bicat$. Explicitly,  $\ob(\bcat{B}/\ecat{X})=\ob(\ecat{X})=X$, while the hom  $(\bcat{B}/\ecat{X})(x,x')$ is given by the slice category  $\bcat{B}(|x|,|x'|)/\ecat{X}(x,x')$ for all $x,x'\in X$.

It follows by Theorem~\ref{thm:Enr-preserves-limits} that  $\enCat{(\bcat{B}/\ecat{X})}$ is the oplax limit 
\[
    \begin{tikzpicture}[baseline=-\the\dimexpr\fontdimen22\textfont2\relax ]
      \node(0) at (0,0) {$\enCat{(\bcat{B}/\ecat{X})}$};
      \node(1) at (2,1) {$\enCat{X_c}$};
      \node(2) at (2,-1){$\enCat{\bcat{B}}$};
      \node(3) at (4,0){$\Setlc$};
      
      \draw [->] (0) to node[auto,labelsize] {} (1);
      \draw [->] (0) to node[auto,swap,labelsize] {} (2);
      \draw [->] (1) to node[auto,labelsize] {$\Enr(\ecat{X})$} (2);
      \draw [->] (1) to node[auto,labelsize] {} (3);
      \draw [->] (1) to node[auto,labelsize] {$\ob$} (3);
      \draw [->] (2) to node[auto,swap,labelsize] {$\ob$} (3);
      \draw [2cell] (1.3,-0.7) to node[auto,swap,labelsize] {} (1.3,0.7);
    \end{tikzpicture}
\]
in $\twoCAT/\Setlc$.

Now $\ob(-)\colon\enCat{X_c}\to\Setlc$ is the free fibration on $X\colon \one\to \Setlc$ by Example~\ref{ex:Xc-Cat}, while $\Enr(\ecat{X})$ is the morphism of fibrations induced by $\ecat{X}\colon \one\to\enCat{\bcat{B}}$ by Proposition~\ref{prop:Enr-fib}, and so by  Proposition~\ref{prop:oplax-limit-fibration} the diagram
\[
    \begin{tikzpicture}[baseline=-\the\dimexpr\fontdimen22\textfont2\relax ]
      \node(0) at (0,0) {$\enCat{(\bcat{B}/\ecat{X})}$};
      \node(1) at (2,1) {$\one$};
      \node(2) at (2,-1){$\enCat{\bcat{B}}$};
      
      \draw [->] (0) to node[auto,labelsize] {} (1);
      \draw [->] (0) to node[auto,swap,labelsize] {} (2);
      \draw [->] (1) to node[auto,labelsize] {$\ecat{X}$} (2);
      \draw [2cell] (1.3,-0.7) to node[auto,swap,labelsize] {} (1.3,0.7);
    \end{tikzpicture}
\]
is an oplax limit in $\twoCAT$. But this says precisely that $\enCat{(\bcat{B}/\ecat{X})}\cong \enCat{\bcat{B}}/\ecat{X}$.
\end{proof}

\begin{example}
\label{ex:Set/X}
    In particular, when $\bcat{B}$ is the cartesian monoidal category $\Set$ regarded as a one-object bicategory, we have for each ($\Set$-)category $\ecat{X}$ the bicategory $\Set/\ecat{X}$ whose set of objects is $\ob(\ecat{X})$ and whose hom-category $(\Set/\ecat{X})(x,x')$ is the slice category $\Set/\ecat{X}(x,x')$. 
    Each functor $F\colon\ecat{Y}\to\ecat{X}$ corresponds to a $\Set/\ecat{X}$-category $\overline{\ecat{Y}}$ given as follows: the objects of $\overline{\ecat{Y}}$ are the same as those of $\ecat{Y}$, the extent of $y$ in $\overline{\ecat{Y}}$ is $Fy$, and the hom $\overline{\ecat{Y}}(y,y')$ is $F_{y,y'}\colon \ecat{Y}(y,y')\to \ecat{X}(Fy,Fy')$.
    Note that since $\Set/\ecat{X}(x,x')\simeq \Set^{\ecat{X}(x,x')}$, $\Set/\ecat{X}$ is (biequivalent to) the free local cocompletion of $\ecat{X}$ regarded as a locally discrete bicategory, as pointed out to us by Ross Street.

    A variant of $\Set/\ecat{X}$ is the free quantaloid $\mathcal{P}\ecat{X}$ over $\ecat{X}$. Specifically, $\mathcal{P}\ecat{X}$ is also a bicategory with the same objects as $\ecat{X}$, but whose hom-category $(\mathcal{P}\ecat{X})(x,x')$ is the powerset lattice $\mathcal{P}(\ecat{X}(x,x'))$, which is equivalent to the full subcategory of the slice category $\Set/\ecat{X}(x,x')$ consisting of the injections to $\ecat{X}(x,x')$. Accordingly, the $\mathcal{P}\ecat{X}$-categories correspond to the \emph{faithful} functors $\ecat{Y}\to\ecat{X}$ \cite[Proposition~3.5]{Garner-total}.
\end{example}

\begin{example}
    Let $\bcat{B}$ be a bicategory with all right liftings. Then for each $b\in \bcat{B}$, we have a $\bcat{B}$-category $\ecat{B}_b$  whose objects are the 1-cells $f\colon x\to b$ in $\bcat{B}$ with codomain $b$, with extent $|(x,f)|=x$, and whose hom $\ecat{B}_b((x,f),(y,g))\colon x\to y$ is the right lifting of $f$ along $g$:
    \[
    \begin{tikzpicture}[baseline=-\the\dimexpr\fontdimen22\textfont2\relax ]
      \node(0) at (0,1) {$x$};
      \node(1) at (3,1) {$y$};
      \node(2) at (1.5,-1) {$b$.};
      \draw [->] (0) to node[auto,labelsize] {$\ecat{B}_b((x,f),(y,g))$} (1);
      \draw[->] (0) to node[auto,swap, labelsize] {$f$} (2);
      \draw[->] (1) to node[auto,labelsize] {$g$} (2);
        \draw [2cell] (2.2,0.3) to node[auto,labelsize] {} (0.8,0.3);
    \end{tikzpicture}    
\]
(See \cite[Section~2]{Gordon_Power} for the dual construction.)
Given a $\bcat{B}$-category $\ecat{X}$, the $\bcat{B}$-functors $\ecat{X}\to\ecat{B}_b$ correspond to the $\bcat{B}$-presheaves on $\ecat{X}$ with extent $b$. Hence if we consider the bicategory $\bcat{B}/\ecat{B}_b$, then a $\bcat{B}/\ecat{B}_b$-category can be identified with a $\bcat{B}$-category equipped with a $\bcat{B}$-presheaf with extent $b$. 

By the universality of right liftings, the bicategory $\bcat{B}/\ecat{B}_b$ is canonically isomorphic to the {\em lax slice} bicategory $\bcat{B}{\sslash} b$: this has 1-cells with codomain $b$ as objects, and diagrams of the form 
   \[
    \begin{tikzpicture}[baseline=-\the\dimexpr\fontdimen22\textfont2\relax ]
      \node(0) at (0,1) {$x$};
      \node(1) at (3,1) {$y$};
      \node(2) at (1.5,-1) {$b$};
      \draw [->] (0) to node[auto,labelsize] {} (1);
      \draw[->] (0) to node[auto,swap, labelsize] {$f$} (2);
      \draw[->] (1) to node[auto,labelsize] {$g$} (2);
        \draw [2cell] (2.2,0.3) to node[auto,labelsize] {} (0.8,0.3);
    \end{tikzpicture}    
\]
as 1-cells from $f$ to $g$. Unlike $\ecat{B}_b$, this lax slice bicategory $\bcat{B}{\sslash} b$ can be defined even when $\bcat{B}$ does not have right liftings, and it is true in general that a $\bcat{B}{\sslash} b$-category corresponds to a $\bcat{B}$-category equipped with a $\bcat{B}$-presheaf with extent $b$. (For a general bicategory $\bcat{B}$, the notion of $\bcat{B}$-presheaf can be defined in terms of actions; see \cite{Street-cohomology} for a definition of the more general notion of module.)
\end{example}

\begin{remark}\label{rmk:tricat-enrichment}
    The bicategory $\bcat{B}/\ecat{X}$ can be obtained from $\ecat{X}$ via a change-of-base process for bicategories enriched in a tricategory. Since the theory of tricategory-enriched bicategories, let alone change-of-base for them, has not really been developed in detail, we merely sketch the details.
    (See \cite[Section 13]{Garner-Shulman} for change-of-base for bicategories enriched over monoidal bicategories.)
    
    We regard $\bcat{B}$ as a tricategory with no non-identity 3-cells, and we regard the cartesian monoidal 2-category $\Cat$ as a one-object tricategory $\Sigma(\Cat)$. There is a lax morphism of tricategories $\Theta\colon\bcat{B}\to\Sigma(\Cat)$ sending each object $b\in\bcat{B}$ to the unique object of $\Sigma(\Cat)$, and sending a 1-cell $f\colon b\to b'$ in $\bcat{B}$ to the category $\bcat{B}(b,b')/f$. Composition with $\Theta$ then sends each $\bcat{B}$-enriched bicategory to a $\Sigma(\Cat)$-enriched bicategory. Since $\bcat{B}$ has no non-identity 3-cells, a $\bcat{B}$-enriched bicategory is just a $\bcat{B}$-enriched category; on the other hand, a $\Sigma(\Cat)$-enriched bicategory is just a bicategory in the ordinary sense. Applying this to the $\bcat{B}$-category $\ecat{X}$ gives the bicategory $\bcat{B}/\ecat{X}$.    
\end{remark}

\section{Variation through enrichment}\label{sect:fibrations}

In the paper \cite{Betti-et-al-variation}, the authors showed how fibrations with codomain $\ecat{X}$ can be seen as certain categories enriched over a bicategory $\bcat{W}(\ecat{X})$ depending on the category $\ecat{X}$. In this section we give a result of the same type, although it differs in several important respects. The bicategory we use is $\Set/\ecat{X}$ (see Example~\ref{ex:Set/X}), which  is like $\bcat{W}(\ecat{X})$ in having as objects the objects of $\ecat{X}$: see Remark~\ref{rmk:W(X)} below for the relationship between the two bicategories. Then we show that fibrations over $\ecat{X}$ can be identified with $\Set/\ecat{X}$-categories which have certain powers. 

\begin{remark}
    \label{rmk:W(X)}
Given objects $x,x'\in\ecat{X}$, a 1-cell in $\bcat{W}(\ecat{X})$ from $x$ to $x'$ consists of a presheaf $E$ on $\ecat{X}$ equipped with maps to $\ecat{X}(-,x)$ and $\ecat{X}(-,x')$; in other words, it consists of a {\em span} of presheaves from $\ecat{X}(-,x)$ to $\ecat{X}(-,x')$. Now a 1-cell $S\to\ecat{X}(x,x')$ in $\Set/\ecat{X}$ from $x$ to $x'$ determines, via Yoneda, a map $S\cdot \ecat{X}(-,x)\to \ecat{X}(-,x')$ of presheaves, where $S\cdot\ecat{X}(-,x)$ denotes the copower of $\ecat{X}(-,x)$ by $S$: the coproduct of $S$ copies of $\ecat{X}(-,x)$. On the other hand there is the codiagonal $S\cdot \ecat{X}(-,x)\to\ecat{X}(-,x)$, and so we obtain a span 
\[
    \begin{tikzpicture}[baseline=-\the\dimexpr\fontdimen22\textfont2\relax ]
      \node(0) at (0,0) {$\ecat{X}(-,x)$};
      \node(1) at (6,0) {$\ecat{X}(-,x')$};
      \node(2) at (3,0) {$S\cdot \ecat{X}(-,x)$};
      \draw [->] (2) to (0);
      \draw [->] (2) to (1);
    \end{tikzpicture}   
\]
of presheaves; that is, a 1-cell in $\bcat{W}(\ecat{X})$ from $x$ to $x'$. This defines the 1-cell part of a homomorphism of bicategories $\Set/\ecat{X}\to\bcat{W}(\ecat{X})$ which is the identity on objects and locally fully faithful. Just as we characterize fibrations over $\ecat{X}$ as $\Set/\ecat{X}$-categories with certain limits, so in \cite{Betti-et-al-variation} these fibrations are seen as $\bcat{W}(\ecat{X})$-categories with certain limits; one key difference is that in the case of $\bcat{W}(\ecat{X})$ the limits in question are absolute. 
\end{remark}

In fact we work not just with fibrations of ordinary categories, but rather fibrations in the 2-category $\enCat{\bcat{B}}$ of  $\bcat{B}$-enriched categories, as in Section~\ref{sect:oplax}. 
One recovers the case of ordinary categories upon taking $\bcat{B}$ to be the one-object bicategory $\Sigma(\Set)$.
We have seen in Theorem~\ref{thm:enrich-over-slice} that, for a $\bcat{B}$-category $\ecat{X}$, $\bcat{B}$-functors with codomain $\ecat{X}$ correspond to $\bcat{B}/\ecat{X}$-enriched categories. We shall see in this section that a $\bcat{B}$-functor $F\colon\ecat{Y}\to\ecat{X}$ is a fibration in  $\enCat{\bcat{B}}$ if and only if the corresponding $\bcat{B}/\ecat{X}$-category $\overline{\ecat{Y}}$ has certain powers. 

First, however, we give an elementary characterization of fibrations in $\enCat{\bcat{B}}$.
To do this, we start with the fact that every $\bcat{B}$-category $\ecat{X}$ has an underlying ordinary category $\ecat{X}_0$ with the same objects; a morphism $x\to x'$ in $\ecat{X}_0$ can exist only if $x$ and $x'$ have the same extent ($|x|=|x'|$), in which case it amounts to a 2-cell $1_{|x|}\to \ecat{X}(x,x')$ in $\bcat{B}$.\footnote{The assignment $\ecat{X}\mapsto\ecat{X}_0$ is the object-part of a 2-functor $\enCat{\bcat{B}}\to\Cat$, arising via change-of-base with respect to a lax functor from $\bcat{B}$ to the cartesian monoidal category $\Set$, seen as a one-object bicategory. The lax functor sends each object $b$ to this unique object; it sends a 1-cell $f\colon b\to c$ to the set $\bcat{B}(b,c)(1_b,f)$ if $b=c$ and the empty set otherwise; with the evident action on 2-cells.} We shall sometimes refer to such morphisms in $\ecat{X}_0$ simply as morphisms in $\ecat{X}$. 
If $f\colon x'\to x''$ is a morphism in $\ecat{X}$ and $x$ is an object, there is an induced 2-cell $\ecat{X}(x,f)\colon\ecat{X}(x,x')\to \ecat{X}(x,x'')$ defined by pasting $f\colon 1_{|x'|}\to\ecat{X}(x',x'')$ together with the composition 2-cell $M_{x,x',x''}\colon\ecat{X}(x',x'').\ecat{X}(x,x')\to \ecat{X}(x,x'')$. 

\begin{definition}
Let $F\colon\ecat{Y}\to\ecat{X}$ be a $\bcat{B}$-functor. A morphism $h\colon y'\to y$ in $\ecat{Y}_0$ is said to be {\em cartesian} with respect to $F$ if the square 
\[
    \begin{tikzpicture}[baseline=-\the\dimexpr\fontdimen22\textfont2\relax ]
      \node(01) at (0,1) {$\ecat{Y}(z,y')$};
      \node(11) at (4,1) {$\ecat{Y}(z,y)$};
      \node(00) at (0,-1){$\ecat{X}(Fz,Fy')$};
      \node(10) at (4,-1){$\ecat{X}(Fz,Fy)$};

      \draw [->] (01) to node[auto,labelsize] {$\ecat{Y}(z,h)$} (11);
      \draw [->] (11) to node[auto,labelsize] {$F_{z,y}$} (10);
      \draw [->] (01) to node[auto,swap,labelsize] {$F_{z,y'}$} (00);
      \draw [->] (00) to node[auto,swap, labelsize] {$\ecat{X}(Fz,Fh)$} (10);
    \end{tikzpicture}
\]
is a pullback in $\bcat{B}(|z|,|y|)$ for all objects $z$ in $\ecat{Y}$. 
\end{definition}

This implies in particular that $h$ is cartesian with respect to the ordinary functor $F_0\colon\ecat{Y}_0\to\ecat{X}_0$, but in general is stronger than this. 

\begin{proposition}\label{prop:fibration-concrete}
Suppose that the bicategory $\bcat{B}$ has pullbacks in each hom-category $\bcat{B}(a,b)$. 
A $\bcat{B}$-functor $F\colon\ecat{Y}\to\ecat{X}$ is a fibration in $\enCat{\bcat{B}}$ if and only if, for each object $y\in\ecat{Y}$ and each morphism $g\colon x\to Fy$ in $\ecat{X}$ there is a cartesian morphism $\overline{g}\colon g^*y\to y$ in $\ecat{Y}$ with $F\overline{g}=g$.
Given fibrations $F\colon \ecat{Y}\to\ecat{X}$ and $G\colon \ecat{Z}\to\ecat{X}$, a $\bcat{B}$-functor $H\colon \ecat{Y}\to\ecat{Z}$ with $F=G\circ H$ is a morphism of fibrations if and only if $H\colon \ecat{Y}\to\ecat{Z}$ preserves cartesian morphisms. 
\end{proposition}

\begin{proof}
The pullbacks in the hom-categories of $\bcat{B}$ can be used to construct oplax limits in $\enCat{\bcat{B}}$, as we shall now show. Let $F\colon\ecat{Y}\to\ecat{X}$ be a $\bcat{B}$-functor; then the oplax limit $\ecat{L}=\ecat{X}/F$ has:
\begin{itemize}
    \item objects given by pairs $(g,y)$, with $y\in\ecat{Y}$ and $g\colon x\to Fy$ in $\ecat{X}_0$
    \item the extent of $(g,y)$ equal to the extent of $y$ (which is also the extent of $x$)
    \item homs given by pullbacks as in 
    \[
    \begin{tikzpicture}[baseline=-\the\dimexpr\fontdimen22\textfont2\relax ]
      \node(01) at (0,1.5) {$\ecat{L}((g',y'),(g,y))$};
      \node(11) at (4,1.5) {$\ecat{Y}(y',y)$};
      \node(00) at (0,-1.5){$\ecat{X}(x',x)$};
      \node(1x) at (4,0) {$\ecat{X}(Fy',Fy)$};
      \node(10) at (4,-1.5){$\ecat{X}(x',Fy)$};

      \draw [->] (01) to node[auto,labelsize] {$U_{(g',y'),(g,y)}$} (11);
      \draw [->] (11) to node[auto,labelsize] {$F_{y',y}$} (1x);
      \draw [->] (1x) to node[auto,labelsize] {$\ecat{X}(g',Fy)$} (10);
      \draw [->] (01) to node[auto,swap,labelsize] {$V_{(g',y'),(g,y)}$} (00);
      \draw [->] (00) to node[auto,swap, labelsize] {$\ecat{X}(x',g)$} (10);
    \end{tikzpicture}
\]
\item projections $V\colon\ecat{L}\to\ecat{X}$ and $U\colon\ecat{L}\to\ecat{Y}$ sending an object $(g,y)$ to $x$ and to $y$, and defined on homs as in the diagram above.
\end{itemize}
The diagonal $\bcat{B}$-functor $D\colon\ecat{Y}\to\ecat{L}$ sends an object $z\in\ecat{Y}$ to $(1_{Fz},z)\in\ecat{L}$. Taking $(g',y')=Dz$ in the above diagram gives a pullback
\[
    \begin{tikzpicture}[baseline=-\the\dimexpr\fontdimen22\textfont2\relax ]
      \node(01) at (0,1) {$\ecat{L}(Dz,(g,y))$};
      \node(11) at (4,1) {$\ecat{Y}(z,y)$};
      \node(00) at (0,-1){$\ecat{X}(Fz,x)$};
      \node(10) at (4,-1){$\ecat{X}(Fz,Fy)$.};

      \draw [->] (01) to node[auto,labelsize] {$U$} (11);
      \draw [->] (11) to node[auto,labelsize] {$F_{z,y}$} (10);
      \draw [->] (01) to node[auto,swap,labelsize] {$V$} (00);
      \draw [->] (00) to node[auto,swap, labelsize] {$\ecat{X}(Fz,g)$} (10);
    \end{tikzpicture}
\]
Now $F$ is a fibration just when $D$ has a right adjoint in $\enCat{\bcat{B}}/\ecat{X}$. Such an adjoint is given on objects by a lifting of $g\colon Fx\to y$ to some $\overline{g}\colon g^* y\to y$, and the universal property says that this lifting is cartesian.
\end{proof}

We now turn to the characterization of fibrations of $\bcat{B}$-categories in terms of $\bcat{B}/\ecat{X}$-categories. First recall that if $\bcat{W}$ is a bicategory and $\ecat{Z}$ is a $\bcat{W}$-category then powers in $\ecat{Z}$ involve an object $y$ of $\ecat{Z}$ and a 1-cell $v\colon x\to |y|$ in $\bcat{W}$ with codomain the extent of $y$. The power of $y$ by $v$ consists of an object $v\pitchfork y$ of $\ecat{Z}$ with extent $|v\pitchfork y|=x$, together with a 2-cell
\[
    \begin{tikzpicture}[baseline=-\the\dimexpr\fontdimen22\textfont2\relax ]
      \node(0) at (0,0) {$|v\pitchfork y|$};
      \node(1) at (3,0) {$|y|$};
      
      \draw [->,bend left=30] (0) to node[auto,labelsize] {} (1);
      \node[auto,labelsize] at (1.5,0.75) {$v$};
      \draw [->,bend right=30] (0) to node[auto,swap,labelsize] {} (1);
      \node[auto,labelsize] at (1.5,-0.8) {$\ecat{Z}(v\pitchfork y,y)$};
      \draw [2cell] (1.5,0.5) to node[auto,labelsize] {$\eta$} (1.5,-0.5);
    \end{tikzpicture}   
\]
such that for all $z\in\ecat{Z}$ and all
\[
    \begin{tikzpicture}[baseline=-\the\dimexpr\fontdimen22\textfont2\relax ]
      \node(0) at (0,0) {$|z|$};
      \node(1) at (1.5,2) {$|v\pitchfork y|$};
      \node(2) at (3,0) {$|y|$};
      \draw [->] (0) to node[auto,labelsize] {$b$} (1);
      \draw[->] (0) to node[auto,swap, labelsize] {$\ecat{Z}(z,y)$} (2);
      \draw[->] (1) to node[auto,labelsize] {$v$} (2);
        \draw [2cell] (1.5,1.5) to node[auto,labelsize] {$\alpha$} (1.5,0.1);
    \end{tikzpicture}    
\]
there exists a unique $\gamma$ making the pasting composite
\[
    \begin{tikzpicture}[baseline=-\the\dimexpr\fontdimen22\textfont2\relax ]
      \node(0) at (0,0) {$|z|$};
      \node(1) at (2,3) {$|v\pitchfork y|$};
      \node(2) at (4,0) {$|y|$};
      \draw [->, bend left=30] (0) to node[auto,labelsize] {$b$} (1);
      \draw [->, bend right=30] (0) to node[pos=0.25,fill=white,labelsize] {$\ecat{Z}(z,v\pitchfork y)$} (1);
      \draw[->,bend right=30] (0) to node[auto,swap,labelsize] {$\ecat{Z}(z,y)$} (2);
      \draw[->,bend left=30] (1) to node[auto,labelsize] {$v$} (2);
      \draw[->,bend right=30] (1) to node[pos=0.75,fill=white,labelsize] {$\ecat{Z}(v\pitchfork y,y)$} (2);
        \draw [2cell] (2,1) to node[auto,labelsize] {$M$} (2,-0.5);
        \draw [2cell] (1,2.2) to node[auto,labelsize] {$\gamma$} (1,0.8);
        \draw [2cell] (3,2.2) to node[auto,labelsize] {$\eta$} (3,0.8);
    \end{tikzpicture}    
\]
equal to $\alpha$.
(In other words, the pasting of $\eta$ and $M$ exhibits $\ecat{Z}(z,v\pitchfork y)$ as the right lifting of $\ecat{Z}(z,y)$ along $v$.)

We consider this in the case where $\bcat{W}=\bcat{B}/\ecat{X}$ and $\ecat{Z}$ is the $\bcat{B}/\ecat{X}$-category $\overline{\ecat{Y}}$ corresponding to a $\bcat{B}$-functor $F\colon\ecat{Y}\to\ecat{X}$. Then an object $y$ of $\overline{\ecat{Y}}$ is just an object of $\ecat{Y}$, and the extent of $y$, as an object of $\bcat{B}/\ecat{X}$, is the object $Fy$ of $\ecat{X}$. A general 1-cell $x\to Fy$ in $\bcat{B}/\ecat{X}$ has the form 
\[
    \begin{tikzpicture}[baseline=-\the\dimexpr\fontdimen22\textfont2\relax ]
      \node(0) at (0,0) {$|x|$};
      \node(1) at (3,0) {$|Fy|$,};
      
      \draw [->,bend left=30] (0) to node[auto,labelsize] {} (1);
      \draw [->,bend right=30] (0) to node[auto,swap,labelsize] {} (1);
      \node[auto,labelsize] at (1.5,0.75) {$v$};
      \node[auto,labelsize] at (1.5,-0.8) {$\ecat{X}(x,Fy)$};
      \draw [2cell] (1.5,0.5) to node[auto,labelsize] {$w$} (1.5,-0.5);
    \end{tikzpicture}   
\]
but we shall only consider the special case where $|x|=|Fy|$ and $v=1_{|x|}$, so that in fact we are dealing with a morphism $w\colon x\to Fy$ in $\ecat{X}_0$. In general, we call a 1-cell $(w\colon v\to \ecat{X}(x,x'))\colon x\to x'$ in $\bcat{B}/\ecat{X}$ a {\em singleton} 1-cell if $|x|=|x'|$ and $v=1_{|x|}$. 
Note that the category $\ecat{X}_0$ can be regarded as a sub-bicategory of $\bcat{B}/\ecat{X}$ whose 1-cells are the singleton 1-cells.
When $\bcat{B}=\Set$, a 1-cell $x\to x'$ in $\Set/\ecat{X}$ corresponds to a set $v$ equipped with a function $w\colon v\to \ecat{X}(x,x')$; in this case, the singleton 1-cells in $\Set/\ecat{X}$ can be identified with those 1-cells with $v$ a singleton, whence the name singleton.

A power of $y$ by $w\colon 1\to \ecat{X}(x,Fy)$ then consists of an object $w\pitchfork y$ of $\ecat{Y}$ with $F(w\pitchfork y)$ $=x$ together with a morphism $\overline{w}\colon w\pitchfork y\to y$ in $\ecat{Y}_0$ with $F\overline{w}=w$ --- that is, a lifting $\overline{w}$ of $w$ --- subject to the universal property stating that for all $z\in\ecat{Y}$, $b\colon |z|\to|y|$, $\alpha$, and $\beta$ making
\[
   \begin{tikzpicture}[baseline=-\the\dimexpr\fontdimen22\textfont2\relax ]
      \node(0) at (0,0) {$|z|$};
       \node(1) at (2,3) {$|x|$};
       \node(2) at (4,0) {$|y|$};
      \draw [->, bend left=30] (0) to node[auto,labelsize] {$b$} (1);
      \draw[->,bend left=40] (1) to node[auto,labelsize] {$1$} (2);
      \draw [->, bend left=30] (0) to node[fill=white,labelsize] {$\ecat{Y}(z,y)$} (2);
      \draw[->,bend right=30] (0) to node[auto,swap,labelsize] {$\ecat{X}(Fz,Fy)$} (2);
      \draw [2cell] (2,2.4) to node[auto,labelsize] {$\alpha$} (2,1);
        \draw [2cell] (2,0.7) to node[auto,labelsize] {$F_{z,y}$} (2,-0.7);
    \end{tikzpicture}    
    \qquad = \qquad 
    \begin{tikzpicture}[baseline=-\the\dimexpr\fontdimen22\textfont2\relax ]
      \node(0) at (0,0) {$|z|$};
      \node(1) at (2,3) {$|x|$};
      \node(2) at (4,0) {$|y|,$};
      \draw [->, bend left=30] (0) to node[auto,labelsize] {$b$} (1);
      \draw [->, bend right=30] (0) to node[pos=0.25,fill=white,labelsize] {$\ecat{X}(Fz,x)$} (1);
      \draw[->,bend right=30] (0) to node[auto,swap,labelsize] {$\ecat{X}(Fz,Fy)$} (2);
      \draw[->,bend left=30] (1) to node[auto,labelsize] {$1$} (2);
      \draw[->,bend right=30] (1) to node[pos=0.75,fill=white,labelsize] {$\ecat{X}(x,Fy)$} (2);
        \draw [2cell] (2,1) to node[auto,labelsize] {$M$} (2,-0.5);
        \draw [2cell] (1,2.2) to node[auto,labelsize] {$\beta$} (1,0.8);
        \draw [2cell] (3,2.2) to node[auto,labelsize] {$w$} (3,0.8);
    \end{tikzpicture}    
\]
there exists a unique $\gamma$ making the pasting composites 
\[
    \begin{tikzpicture}[baseline=-\the\dimexpr\fontdimen22\textfont2\relax ]
      \node(0) at (0,0) {$|z|$};
       \node(1) at (2,3) {$|w\pitchfork y|$};
       \node(2) at (4,0) {$|y|$};
      \draw [->, bend left=30] (0) to node[auto,labelsize] {$b$} (1);
      \draw[->,bend left=40] (1) to node[auto,labelsize] {$1$} (2);
      \draw [->, bend left=30] (0) to node[fill=white,labelsize] {$\ecat{Y}(z,w\pitchfork y)$} (2);
      \draw[->,bend right=30] (0) to node[auto,swap,labelsize] {$\ecat{X}(Fz,x)$} (2);
        \draw [2cell] (2,2.4) to node[auto,labelsize] {$\gamma$} (2,1);
        \draw [2cell] (2,0.7) to node[auto,labelsize] {$F_{z,w\pitchfork y}$} (2,-0.7);
    \end{tikzpicture}    
    \qquad  \qquad    \begin{tikzpicture}[baseline=-\the\dimexpr\fontdimen22\textfont2\relax ]
      \node(0) at (0,0) {$|z|$};
      \node(1) at (2,3) {$|w\pitchfork y|$};
      \node(2) at (4,0) {$|y|$};
      \draw [->, bend left=30] (0) to node[auto,labelsize] {$b$} (1);
      \draw [->, bend right=30] (0) to node[pos=0.25,fill=white,labelsize] {$\ecat{Y}(z,v\pitchfork y)$} (1);
      \draw[->,bend right=30] (0) to node[auto,swap,labelsize] {$\ecat{Y}(z,y)$} (2);
      \draw[->,bend left=30] (1) to node[auto,labelsize] {$1$} (2);
      \draw[->,bend right=30] (1) to node[pos=0.75,fill=white,labelsize] {$\ecat{Y}(w\pitchfork y,y)$} (2);
        \draw [2cell] (2,1) to node[auto,labelsize] {$M$} (2,-0.5);
        \draw [2cell] (1,2.2) to node[auto,labelsize] {$\gamma$} (1,0.8);
        \draw [2cell] (3,2.2) to node[auto,labelsize] {$\overline{w}$} (3,0.8);
    \end{tikzpicture}    
\]
equal respectively to $\beta$ and $\alpha$. But this says exactly that if the exterior of the diagram 
\[
    \begin{tikzpicture}[baseline=-\the\dimexpr\fontdimen22\textfont2\relax ]
      \node(01) at (0,1) {$\ecat{Y}(z,w\pitchfork y)$};
      \node(11) at (4,1) {$\ecat{Y}(z,y)$};
      \node(00) at (0,-1){$\ecat{X}(Fz,x)$};
      \node(10) at (4,-1){$\ecat{X}(Fz,Fy)$};
      \node(x) at (-2,3) {$b$};

      \draw [->] (01) to node[auto,labelsize] {$\ecat{Y}(z,\overline{w})$} (11);
      \draw [->] (11) to node[auto,labelsize] {$F_{z,y}$} (10);
      \draw [->] (01) to node[auto,swap,labelsize] {$F_{z,w\pitchfork y}$} (00);
      \draw [->] (00) to node[auto,swap, labelsize] {$\ecat{X}(Fz,w)$} (10);
      \draw[->, bend left=20] (x) to node[auto,labelsize] {$\alpha$} (11);
      \draw[->, bend right=28] (x) to node[auto,swap,labelsize] {$\beta$} (00);
      \draw[->,dotted] (x) to node[auto,labelsize] {$\gamma$} (01);
    \end{tikzpicture}
\]
in $\bcat{B}(|z|,|y|)$ commutes, then there is a unique $\gamma$ making the triangular regions commute; in other words, that the internal square is a pullback. This in turn says that $\overline{w}$ is a cartesian lifting of $w$. 
This now proves:

\begin{proposition}
Let $\bcat{B}$ be a bicategory in which each hom-category has pullbacks. A $\bcat{B}$-functor $F\colon\ecat{Y}\to\ecat{X}$ is a fibration if and only if the corresponding $\bcat{B}/\ecat{X}$-category $\overline{\ecat{Y}}$ has powers by morphisms in $\ecat{X}_0$; that is, powers by singleton 1-cells.  
\end{proposition}

We conclude by strengthening this correspondence to an isomorphism between suitable 2-categories. 
Let $\bcat{W}$ be a bicategory and $H\colon \ecat{Z}\to\ecat{Z}'$ a $\bcat{W}$-functor. Suppose that the power $v\pitchfork y$ of $y\in\ecat{Z}$ by $v\colon x\to|y|$ exists in $\ecat{Z}$, with the associated 2-cell $\eta\colon v\to \ecat{Z}(v\pitchfork y,y)$. Then $H$ is said to \emph{preserve} the power $v\pitchfork y$ if the 2-cell $H_{v\pitchfork y,y}\circ \eta\colon v\to \ecat{Z}'(H(v\pitchfork y),Hy)$ exhibits $H(v\pitchfork y)$ as the power $v\pitchfork Hy$ in $\ecat{Z}'$. 

\begin{theorem}
Let $\bcat{B}$ be a bicategory in which each hom-category has pullbacks. The canonical isomorphism of 2-categories $\enCat{(\bcat{B}/\ecat{X})}\cong \enCat{\bcat{B}}/\ecat{X}$ restricts to an isomorphism between the locally full sub-2-category of $\enCat{(\bcat{B}/\ecat{X})}$ having  
\begin{itemize}
    \item the $\bcat{B}/\ecat{X}$-categories with powers by singleton 1-cells as objects, and 
    \item the $\bcat{B}/\ecat{X}$-functors preserving these powers as 1-cells,
\end{itemize} 
and the locally full sub-2-category of $\enCat{\bcat{B}}/\ecat{X}$ having 
\begin{itemize}
    \item the fibrations to $\ecat{X}$ as objects, and 
    \item the fibration morphisms as 1-cells. 
\end{itemize}
\end{theorem}


\begin{thebibliography}{BCSW83}

\bibitem[Bak]{Bakovic}
Igor Bakovi{\'c}.
\newblock Fibrations of bicategories.
\newblock Preprint available at
  \url{http://www.irb.hr/korisnici/ibakovic/groth2fib.pdf}.

\bibitem[BCSW83]{Betti-et-al-variation}
Renato Betti, Aurelio Carboni, Ross Street, and Robert Walters.
\newblock Variation through enrichment.
\newblock {\em J. Pure Appl. Algebra}, 29(2):109--127, 1983.

\bibitem[B{\'{e}}n67]{Benabou-bicat}
Jean B{\'{e}}nabou.
\newblock Introduction to bicategories.
\newblock In {\em Reports of the {M}idwest {C}ategory {S}eminar}, pages 1--77.
  Springer, Berlin, 1967.

\bibitem[Buc14]{Buckley}
Mitchell Buckley.
\newblock Fibred 2-categories and bicategories.
\newblock {\em J. Pure Appl. Algebra}, 218(6):1034--1074, 2014.

\bibitem[Gar14]{Garner-total}
Richard Garner.
\newblock Topological functors as total categories.
\newblock {\em Theory Appl. Categ.}, 29(15):406--422, 2014.

\bibitem[GP97]{Gordon_Power}
R.~Gordon and A.~J. Power.
\newblock Enrichment through variation.
\newblock {\em J. Pure Appl. Algebra}, 120(2):167--185, 1997.

\bibitem[Gra74]{Gray-FCT}
John~W. Gray.
\newblock {\em Formal category theory: adjointness for {$2$}-categories}.
\newblock Lecture Notes in Mathematics, Vol. 391. Springer-Verlag, Berlin-New
  York, 1974.
  
\bibitem[GS16]{Garner-Shulman}
Richard Garner and Michael Shulman.
\newblock Enriched categories as a free cocompletion.
\newblock {\em Adv. Math.}, 289:1--94, 2016.

\bibitem[Her99]{Hermida}
Claudio Hermida.
\newblock Some properties of {${\bf Fib}$} as a fibred {$2$}-category.
\newblock {\em J. Pure Appl. Algebra}, 134(1):83--109, 1999.

\bibitem[Lac05]{lax}
Stephen Lack.
\newblock Limits for lax morphisms.
\newblock {\em Appl. Categ. Structures}, 13(3):189--203, 2005.

\bibitem[Lac10]{Lack_icons}
Stephen Lack.
\newblock Icons.
\newblock {\em Appl. Categ. Structures}, 18(3):289--307, 2010.

\bibitem[LS12]{LackShulman}
Stephen Lack and Michael Shulman.
\newblock Enhanced 2-categories and limits for lax morphisms.
\newblock {\em Adv. Math.}, 229(1):294--356, 2012.

\bibitem[Str74]{Street-YonedaFibrations}
Ross Street.
\newblock Fibrations and {Y}oneda's lemma in a {$2$}-category.
\newblock In {\em Category {S}eminar ({P}roc. {S}em., {S}ydney, 1972/1973)},
  Lecture Notes in Math., Vol. 420, pages 104--133. Springer, Berlin, 1974.

\bibitem[Str83]{Street-cohomology}
Ross Street.
\newblock Enriched categories and cohomology.
\newblock {\em Repr. Theory Appl. Categ.}, (14):1--18, 2005.
\newblock Reprinted from Quaestiones Math. {{\bf{6}}} (1983), no. 1-3,
  265--283, with new commentary by the author.

\bibitem[Web07]{Weber-Yoneda}
Mark Weber.
\newblock Yoneda structures from 2-toposes.
\newblock {\em Appl. Categ. Structures}, 15(3):259--323, 2007.

\end{thebibliography}

\end{document}